\documentclass{amsart}
\usepackage{amsmath,amsfonts,amssymb,amsthm,cite,amscd,epsf,epic}\usepackage{eepic}

\makeatletter
\renewcommand\thetable{\thesection.\@arabic\c@table}
\renewcommand\thefigure{\thesection.\@arabic\c@figure}
\makeatother
\renewcommand\Re{\operatorname*{Re}}	
\newcommand\res{\operatorname{res}}
\newcommand\ord{\operatorname{ord}}
\newcommand\vol{\operatorname{vol}}
\newcommand\Z{\mathbb{Z}}	\newcommand\C{\mathbb{C}}
\newcommand\Q{\mathbb{Q}}	\newcommand\R{\mathbb{R}}
\newcommand\Gal{\mathrm{Gal}}
\newcommand\spec{\operatorname{spec}}
\renewcommand\P{\mathbb{P}}
\newcommand\F{\mathbb{F}}
\newcommand\Fpa{\bar\F_p}
\newcommand\q{\mathbf{q}}
\newcommand\oo{\mathcal{O}}
\newcommand\Curve{\mathcal{C}}
\newcommand\Mellin{\mathcal{M}}
\newcommand\Tr{\operatorname{Tr}}
\newcommand\pr{\mathfrak{p}}
\newcommand\diff{\mathfrak{d}}
\newcommand\Fo{\mathcal{F}}
\newcommand\A{\mathbb{A}}
\newcommand\x{\mathfrak{x}}	\newcommand\y{\mathfrak{y}}
\newcommand\calD{\mathcal{D}}	\newcommand\calK{\mathcal{K}}
\newcommand\calP{\mathcal{P}}
\newcommand\Pic{\mathrm{Cl}}
\newcommand\gota{\mathfrak{a}}	
\newcommand\cid{\mathfrak{c}}

\newtheorem{theorem}{Theorem}[section]
\newtheorem{lemma}[theorem]{Lemma}

\theoremstyle{definition}\newtheorem{definition}[theorem]{Definition}
\theoremstyle{remark}\newtheorem{remark}[theorem]{Remark}
\newtheorem{example}[theorem]{Example}
\numberwithin{equation}{section}

\begin{document}

\title[The Riemann Hypothesis for Function Fields over a Finite Field]{The Riemann Hypothesis\\
for Function Fields over a Finite Field}
\author{Machiel van Frankenhuijsen}
\address{Department of Mathematics,
Utah Valley University,
Orem, UT 84058-5999}
\email{vanframa@uvsc.edu}


\begin{abstract}
The Riemann hypothesis,
formulated in 1859 by Bernhard Riemann,
states that the Riemann zeta function $\zeta(s)$ has all its nonreal zeros on the line $\Re s=1/2$.
Despite over a hundred years of considerable effort by numerous mathematicians,
this conjecture remains one of the most intriguing unsolved problems in mathematics.
On the other hand,
several analogues of the Riemann hypothesis have been formulated and proved.

In this chapter,
we discuss Enrico Bombieri's proof of the Riemann hypothesis for a curve over a finite field.
This problem was formulated as a conjecture by Emil Artin in his thesis of 1924.
Reformulated,
it states that the number of points on a curve $\Curve$ defined over the finite field with $q$ elements is of the order~$q+O(\sqrt{q})$.
The first proof was given by Andr\'e Weil in 1942.
This proof uses the intersection of divisors on $\Curve\times\Curve$,
making the application to the original Riemann hypothesis so far unsuccessful,
because $\spec\Z\times\spec\Z=\spec\Z$ is one-dimensional.
A new method of proof was found in 1969 by S.~A.~Stepanov.
This method was greatly simplified and generalized by Bombieri in 1973.

Bombieri's proof uses functions on $\Curve\times\Curve$,
again precluding a direct translation to a proof of the Riemann hypothesis itself.
However,
the two coordinates on $\Curve\times\Curve$ play different roles,
one coordinate playing the geometric role of the variable of a polynomial,
and the other coordinate the arithmetic role of the coefficients of this polynomial.
The Frobenius automorphism of $\Curve$ acts on the geometric coordinate of $\Curve\times\Curve$.
In the last section,
we make some suggestions how Nevanlinna theory could provide a model for $\spec\Z\times\spec\Z$ that is two-dimensional and carries an action of Frobenius on the geometric coordinate.

The plan of this chapter is as follows.
We first give a historical introduction to the Riemann hypothesis for a curve over a finite field and discuss some of the proofs that have been given.
In Section~\ref{S:zeta},
we define the zeta function of the curve.
To prove the functional equation,
we take a brief excursion to the two-variable zeta function of Pellikaan.
In Section~\ref{S:RH},
we reformulate the Riemann hypothesis for $\zeta_\Curve(s)$,
 and give Bombieri's proof.
In the last section,
we compare $\zeta_\Curve(s)$ with the Riemann zeta function $\zeta(s)$,
and describe the formalism of Nevanlinna theory that might be a candidate for the framework of a translation of Bombieri's proof to the original Riemann hypothesis.

\end{abstract}
\maketitle

\section{Introduction}
\label{S: intro}

The Riemann zeta function is the function $\zeta(s)$,
defined for $\Re s>1$ by the infinite series
\begin{gather}\label{D:Rzeta}
\zeta(s)=1+\frac1{2^s}+\frac1{3^s}+\frac1{4^s}+\dots.
\end{gather}
This function has a meromorphic continuation to the complex plane,
with a simple pole at~$s=1$ with residue $1$.
One can complete this function with the factor `at infinity' (related to the archimedean valuation on the real numbers),
$$
\zeta_\R(s)=\pi^{-s/2}\Gamma(s/2),
$$
to define the function~$\zeta_\Z(s)=\zeta_\R(s)\zeta(s)$.
This function is meromorphic on~$\C$ with simple poles at $s=0$ and $1$.
It satisfies the functional equation
$$
\zeta_\Z(1-s)=\zeta_\Z(s).
$$
This is proved by Riemann~\cite{Riemann1859} using the `Riemann--Roch'\footnote
{This formula goes back to Cauchy and is called a Riemann--Roch formula in Tate's thesis.
The classical theta function is defined as $\smash{\vartheta(z,\tau)=\sum_{n=-\infty}^\infty e^{\pi in^2\tau}e^{2\pi iz}}$,
so that~\mbox{$\theta(t)=\vartheta(0,it^2)$}.}
 formula
\begin{gather}\label{E: theta flip}
\theta(t^{-1})=t\theta(t),
\end{gather}
where $\theta(t)=\sum_{n=-\infty}^\infty e^{-\pi n^2t^2}$
 is closely related to the classical theta-function.
For $\Re s>1$,
the zeta function satisfies the Euler product
\begin{gather}\label{E: Euler prod}
\zeta_\Z(s)=\zeta_\R(s)\prod_p\frac1{1-p^{-s}},
\end{gather}
where the product is taken over all prime numbers.
It follows that the zeros of~$\zeta_\Z$ all lie in the vertical strip $0\leq\Re s\leq1$.\footnote
{It is also known that the zeros do not lie on the boundary of this `critical strip',
see \cite[Theorem~19, p.~58]{Ingham} and~\cite[Theorem~2.4]{zeros}.}
The Riemann hypothesis states that these zeros all lie on the line $\Re s=1/2$:
$$
\text{Riemann hypothesis:}\quad\zeta_\Z(s)=0\text{ implies }\Re s=1/2.
$$
See~\cite{Riemann1859} and~\cite{Edwards,Haran,book,second,Patterson,RiemannT,Tate} for more information about the Riemann and other zeta functions.\medskip

In this exposition,
we prove the Riemann hypothesis for the zeta function of a curve over a finite field.
Let $q$ be a power of a prime number~$p$,
and let $m(T,X)$ be a polynomial in two variables with coefficients in $\F_q$,
the finite field with $q$ elements.
The equation
$$
\Curve\colon m(T,X)=0
$$
 defines a curve $\Curve$ over $\F_q$,
which we assume to be nonsingular.
Let~$N_\Curve(n)$ be the number of solutions of the equation $m(t,x)=0$ in the finite set~\mbox{$\F_{q^n}\times\F_{q^n}$}.
Thus $N_\Curve(n)$ is the number of points on $\Curve$ with coordinates in $\F_{q^n}$.
A famous theorem of F.\ K.~Schmidt of 1931 (see~\cite{fSchmidt},
\cite{Hasse34,Hassebook,Tretkoff},
and~\eqref{E:NCn} below) says that there exist an integer~$g$,
the genus of~$\Curve$,
and algebraic numbers $\omega_1,\dots,\omega_{2g}$,
such that\footnote
{In $N_\Curve(n)$,
also the finitely many points `at infinity' need to be counted.}
$$
N_\Curve(n)=q^n-\sum_{\nu=1}^{2g}\omega_\nu^n\ +1.
$$
From the formula for $N_\Curve(n)$,
 the Mellin transform (generating function),
$$
\Mellin N_\Curve(s)=\sum_{n=1}^\infty N_\Curve(n)q^{-ns},
$$
can be computed as a rational function of $q^{-s}$,
\begin{gather}\label{E:MNC}
\Mellin N_\Curve(s)+2-2g=\frac1{1-q^{1-s}}-\sum_{\nu=1}^{2g}\frac1{1-\omega_\nu q^{-s}}+\frac1{1-q^{-s}}.
\end{gather}
We define the {\em zeta function\/} of $\Curve$ by
$$
\zeta_\Curve(s)=q^{s(g-1)}
\frac{\prod_{\nu=1}^{2g}(1-\omega_\nu q^{-s})}{(1-q^{1-s})(1-q^{-s})},
$$
so that $\Mellin N_\Curve$ is recovered as its logarithmic derivative,
\begin{gather}\label{E:ld zetaC}
-\frac1{\log q}\frac{\zeta_\Curve'(s)}{\zeta_\Curve(s)}=\Mellin N_\Curve(s)+1-g.
\end{gather}
The function $\zeta_\Curve$ satisfies the functional equation $\zeta_\Curve(1-s)=\zeta_\Curve(s)$.
This functional equation can be proved using the Riemann--Roch formula
\begin{gather}\label{E:RR}
l(\calD)=\deg\calD+1-g+l(\calK-\calD),
\end{gather}
which is the analogue of~\eqref{E: theta flip} above.
We will take a different approach and prove the functional equation in Section~\ref{S:two-variable} using the two-variable zeta function of Pellikaan.\medskip

Clearly,
$\zeta_\Curve$ is a rational function in $q^{-s}$,
and hence periodic with period~$2\pi i/\log q$.
It has simple poles at $s=2k\pi i/\log q$ and $s=1+2k\pi i/\log q$,
and zeros at the points $s=\log_q\omega_\nu+2k\pi i/\log q$ ($\nu=1,\dots,2g$,
$k\in\Z$).
It satisfies an Euler product,
analogous to~\eqref{E: Euler prod},
which converges for $\Re s>1$,
$$
\zeta_\Curve(s)=q^{s(g-1)}\prod_v\frac1{1-q^{-s\deg v}},
$$
where the product is taken over all valuations of the function field of $\Curve$.
It follows that $1\leq|\omega_\nu|\leq q$.
Artin\footnote{In his thesis~\cite{ArtinThesis},
Artin only considers quadratic extensions of $\F_p(T)$,
that is,
hyperelliptic curves over $\F_p$.
Moreover,
in his zeta functions the Euler factors corresponding to the points at infinity are missing.
Later,
F.\ K.~Schmidt~\cite{fSchmidt} introduced the zeta function of a general projective curve over an arbitrary finite field.}
conjectured that~$\zeta_\Curve(s)$ has its zeros on the line $\Re s=1/2$.
In terms of the exponentials of the zeros,
the numbers~$\omega_\nu$,
this means that
$$
\text{Riemann hypothesis for $\Curve$:}\quad|\omega_\nu|=\sqrt q\ \text{ for every }\nu=1,\dots,2g.
$$
It is the analogue of the Riemann hypothesis for $\zeta_\Curve$.
It is trivially verified for $\Curve=\P^1$,
when $g=0$ and $\zeta_\Curve$ does not have any zeros.
It was proved by H.~Hasse in the case of elliptic curves ($g=1$),
and first in full generality by A.~Weil,\footnote
{Weil announced his ideas in 1940 \cite{aW40,aW41} and explained them in 1942 in a letter to Artin~\cite{letter}.
But the complete proof (see \cite{aW49,aW71}) had to await the completion of his `Foundations'~\cite{Foundations}.}
 using the intersection of divisors with the graph of Frobenius in $\Curve\times\Curve$,
and,
in his second proof,
 the action of Frobenius on the embedding of $\Curve$ in its Jacobian (see~\cite[Appendix]{Rosen}).
Later proofs,
based on one of Weil's proofs,
have been given by P.~Roquette~\cite{Roquette} and others.
There have been some attempts to translate the first and second proof to the situation of the Riemann zeta function,
when~$\Curve$ is the `curve'~$\spec\Z$,
but so far without success,
one of the obstacles being that in the category of schemes,
$\spec\Z\times\spec\Z=\spec\Z$ is one-dimensional and not two-dimensional like~$\Curve\times\Curve$ (see~\cite{indextheory}).

A completely new technique was discovered by Stepanov~\cite{Ste} in 1969,
initially only for hyperelliptic curves.
W.\ M.~Schmidt~\cite{wSchmidt} used his method to reprove the Riemann hypothesis for $\zeta_\Curve$,\footnote
{The extension of Stepanov's proof to a general $\F_q$ by W.\ M.\ Schmidt uses hyperderivations.}
and a simplified proof was given by Bombieri in~\cite{Bo} (see also~\cite{BoP}).
Bombieri's proof uses the graph of Frobenius in~\mbox{$\Curve\times\Curve$}
and the Riemann--Roch formula.\footnote
{Bombieri does not mention $\Curve\times\Curve$,
but he uses $\Fpa(\Curve)\otimes\Fpa(\Curve)$,
i.e.,
functions on $\Curve\times\Curve$.}
When Bombieri was studying Stepanov's proof,
 in order to understand the derivations in a geometric way,
he used a standard connection on the curve.
The argument worked only over the prime field because of the obstruction arising from iterating the connection $p$ times.
However,
Bombieri found that derivatives of order $p$ gave the Cartier operator,
which turned out to be the same as taking the $p$-th power of the function.
Thus the final proof does not mention connections nor the Cartier operator,
 and fits on a paper napkin \cite{Bopersonal}.
We present this proof in Section~\ref{S:RH}.
A shorter exposition of Bombieri's proof has appeared in \cite{Expositiones}.\medskip

It is interesting to see the development in these proofs.
Gradually,
more geometry that cannot be translated to the number field case has been taken out.
The question arises as to what exactly is needed to prove the Riemann hypothesis for curves,
and what we can learn from this about the Riemann hypothesis for $\spec\Z$.

Weil's first proof uses the geometry of $\Curve\times\Curve$,
and in particular the intersection of the graph of Frobenius with the diagonal.
There is some reason to believe that no analogue will ever exist for number fields,
or else,
that constructing an analogue is harder than establishing the Riemann hypothesis.
His second proof uses the Jacobian of~$\Curve$,
and again,
no analogue may ever be constructed for the integers.

Bombieri's proof uses very little of the geometry of $\Curve\times\Curve$,
 but it uses the action of Frobenius and Riemann--Roch.
Since Tate's thesis,
it is known that the Riemann--Roch equality~\eqref{E:RR} translates into formula~\eqref{E: theta flip}.
Bombieri's proof naturally divides into two steps.
In the first step,
he uses the action of Frobenius to obtain a discrete flow on the curve (here called the {\em Frobenius flow\/}),
which he analyzes to obtain an upper bound for the number of points on the curve,
which is a weak form of the prime number theorem for the curve (see Table~\ref{T: comparison}).
One sees that the horizontal coordinate of $\Curve\times\Curve$ plays an `arithmetic' role,
and the vertical coordinate plays a `geometric' role (see Remark~\ref{R:arith geo}).
In the second step,
he uses the Riemann hypothesis for~$\P^1$,
and the fact that the Frobenius automorphism generates the local Galois group (decomposition group) at a point on the curve,
to obtain a lower bound for the number of points on the curve.
Combining the two steps,
he first obtains the analogue of the prime number theorem with a good error term,
and from this it is a small step to deduce the Riemann hypothesis (see Lemma~\ref{L: RH}).
Therefore,
one might conclude that the right approach to the Riemann hypothesis is to first prove the prime number theorem with a good bound for the error term.
Looking at the first step of Bombieri's proof,
one might even believe that the key to a prime number theorem with a good error bound is to construct a function (possibly a Fourier or Dirichlet polynomial,
in the spirit of the methods of Baker,
Gelfond and Schneider) that vanishes at the first $N$ primes to a high order.
If one could bound the degree of this polynomial,
 one would obtain an upper bound for the number of primes.
This would already imply the Riemann hypothesis,
so the second step becomes unnecessary.
In Section~\ref{S: speculations} we make some speculations about how to translate Bombieri's proof to number theory.

\begin{remark}
Recently,
Alain Connes found a completely new method again~\cite{Connes},
based on the work of Shai Haran~\cite{Haran,mystery},
using harmonic analysis on the ring of adeles.
Connes conjectures that Weil's explicit formula is obtained as the trace of a certain shift operator,
and then establishes the positivity of this trace,
thus proving the Riemann hypothesis for $\spec\Z$ and for all L-functions associated with a Gr\"ossencharakter,
up to a `lemma' about special functions.
He does not use the action of Frobenius.
(See however Remark~\ref{R: shift}.)
\end{remark}

\section{The Zeta Function of a Curve over a Finite Field}
\label{S:zeta}

Throughout,
we fix a function field $K$ of transcendence degree one over a finite field of characteristic $p$.
We assume that the algebraic closure of $\F_p$ in $K$,
i.e.,
 the field of constants of $K$,
is $\F_q$,
 and we choose a function $T\in K$ such that $K$ is a finite separable extension of the field of rational functions $\F_q(T)$,
which we denote by $\q$,
$$
\q=\F_q(T).
$$
Thus,
$$
K=\q[X]/(m(T,X))
$$
for some polynomial $m$,
irreducible over the algebraic closure $\Fpa$.

\subsubsection*{The Frobenius Flow}

Geometrically,
$K$ is the field of functions on a curve $\Curve$ given by $m(T,X)=0$,
and the choice of $T$ corresponds to the choice of a projection of $\Curve$ onto $\P^1$.
Adding extra coordinates $T,X,Y,\dots$ if necessary,
we can obtain a nonsingular model for $\Curve$.

Valuations of $K$ correspond to orbits of points on $\Curve$ in the following way.
The Frobenius automorphism acts on algebraic points in $\Curve(\Fpa)$ by raising each coordinate of a point to the $q$-th power,
$$
\phi_q\colon(t,x,y,\dots)\mapsto(t^q,x^q,y^q,\dots).
$$
This action of Frobenius induces a discrete dynamical system on $\Curve(\Fpa)$,
which we call the {\em Frobenius flow\/} on $\Curve$ (see Figure~\ref{F: Frobenius}).
So a point in $\Curve(\F_{q^n})$ that is not defined over a smaller field of constants gives an orbit of this flow of $n$ points.
We normalize the corresponding valuation of $K$ so that $v(f)$ equals the order of vanishing of $f$ at any one of the points in this orbit.

\subsection{The Local Theory}
\label{S: local}

For the local theory of the zeta function,
our exposition closely follows Tate's thesis~\cite{Tate}.
The completion of $K$ at a valuation $v$ is denoted~$K_v$.
As is well known,
$K_v$ is locally compact.
It contains the ring $\oo_v$ of function elements regular at $v$,
 which has a single prime ideal
$$
\pi_v\oo_v
$$
 of functions vanishing at $v$.
Here,
$\pi_v$ is any function that has a simple zero at $v$.
The residue class field
$$
K(v)=\oo_v/\pi_v\oo_v
$$
 is a finite extension of~$\F_q$.
We denote its degree over $\F_q$ by $\deg v$.
This is also the length of the corresponding orbit of the Frobenius flow.
We write
$$
q_v=q^{\deg v}=|K(v)|
$$
 for the cardinality of $K(v)$.
By our normalization of the valuations of $K$,
we have~$v(\pi_v)=1$.
The associated norm is
$$
|\alpha|_v=q_v^{-v(\alpha)}.
$$

\subsubsection{Additive Characters and Measure}

Denote by $K_v^+$ the additive group of $K_v$,
as a locally compact commutative group,
and by $\xi$ its general element.
The character group of $K_v^+$ is determined by the following lemma:

\begin{lemma}[{[\citen{Tate}, Lemma~2.2.1]}]\label{L:chars}
If\/ $\xi\mapsto\chi(\xi)$ is one nontrivial character of\/ $K_v^+,$
then for each\/~$\eta\in K_v^+,$
$\xi\mapsto\chi(\eta\xi)$ is also a character of\/ $K_v^+$.
The correspondence $\eta\longleftrightarrow\chi(\eta\xi)$ is an isomorphism,
both topological and algebraic,
between $K_v^+$ and its character group.
\end{lemma}

To fix the identification of $K_v^+$ with its character group promised by this lemma,
we must construct a special nontrivial character.
We first construct additive characters for $\q=\F_q(T)$.
The restriction of $v$ to $\q$ is either a multiple of a $P$-adic valuation,
for an irreducible polynomial $P$,
or a multiple of $v_\infty$,
the valuation at infinity,
corresponding to $P(T)=1/T$.
Let~$\q_P$ be the completion of~$\q$ at $P$.
Thus each element of $\q_P$ is a Laurent series of terms $aT^kP^n$,
for $a\in\F_q$,
$0\leq k<\deg P$ ($k=0$ if $P=1/T$),
with only finitely many terms with a negative power of~$P$.

Recall that the characteristic of~$\q$ is $p$.
We identify~$\F_p$ with~$\Z/p\Z$,
so that for $t\in\F_p$,
the rational number~$t/p$ is well defined modulo~$\Z$.
Define a character $\chi_P$ on $\q_P$ as follows:
If $P(T)=1/T$,
then
\begin{gather}\label{E:chiinfty}
\chi_\infty(aT^n)=\begin{cases}
1			&\text{if }n\neq-1\\
\exp\Bigl(-\frac{2\pi i}p\Tr_{\F_q/\F_p}(a)\Bigr)&\text{if }n=-1,
\end{cases}
\end{gather}
where $\Tr_{\F_q/\F_p}(a)$ denotes the trace of $a\in\F_q$ over $\F_p=\Z/p\Z$.
If $P(T)$ is an irreducible monic polynomial of degree~$d$ with coefficients in $\F_q$,
then we put for $0\leq k\leq d-1$ and $a\in\F_q$,
\begin{gather}\label{E:chiP}
\chi_P(aT^kP^n)=\begin{cases}
1			&\text{if $n\neq-1$ or }0\leq k\leq d-2\\
\exp\Bigl(\frac{2\pi i}p\Tr_{\F_q/\F_p}(a)\Bigr)	&\text{if $n=-1$ and }k=d-1.
\end{cases}
\end{gather}
See \cite[Lemma~2.2.2]{Tate} for a proof of the following lemma.

\begin{lemma}
$\xi\mapsto\chi_P(\xi)$ is a nontrivial,
continuous additive map of\/ $\q_P$ into the unit circle group.
\end{lemma}

Unlike in the number field case,
it is not true that $\chi_P(\xi)=1$ if and only if $\xi$ is a $P$-adic integer.
In fact,
$\chi_P$ depends only on the coefficient of $1/P$ (or of $1/T$ at the infinite valuation).
We will see the geometric meaning of~$\chi_P$ and $\chi_\infty$ by relating it to a residue.

\begin{lemma}\label{L: trace=delta}
Let\/ $P$ be an irreducible monic polynomial of degree $d$ with coefficients in~$\F_q$.
Let\/ $\q(P)$ be the residue class field\/ $\F_q[T]/(P)$.
Then
$$
\Tr_{\q(P)/\F_q}\left(\frac{T^k}{P'(T)}+(P)\right)=\begin{cases}
0	&\text{if\/ }0\leq k\leq d-2\\
1	&\text{if\/ }k=d-1.
\end{cases}
$$
\end{lemma}

\begin{proof}
Let $t_1,\dots,t_d$ be the roots of $P$ in $\q(P)$.
Then,
$$
\Tr_{\q(P)/\F_q}\left(\frac{T^k}{P'(T)}+(P)\right)=\sum_{i=1}^d\frac{t_i^k}{P'(t_i)}.
$$
We can lift $P$ to a polynomial $\widetilde P$ over a number field $F$ that reduces to~$P$ modulo a prime ideal $\pr$ of $F$ above $p$ with residue class field $\F_q$.
The roots of~$\smash{\widetilde P}$ then reduce to the roots $t_i$ modulo a prime ideal above $\pr$ in the splitting field of $\smash{\widetilde P}$.
Since the roots of $P$ are all distinct,
so are those of~$\smash{\widetilde P}$.
The lemma now follows from the following more general lemma,
which asserts that the required equality already holds without taking the class in~$\F_q$.
\end{proof}

\begin{lemma}\label{L: residues}
Let\/ $m(X)$ be a monic polynomial over\/ $\C$ of degree\/~$d$ without repeated roots.
Let\/ $a_1,\dots,a_d\in\C$ be the roots of\/ $m$.
Then
$$
\sum_{i=1}^d\frac{a_i^k}{m'(a_i)}=\begin{cases}0&\text{for }0\leq k\leq d-2,\\
1&\text{for }k=d-1.
\end{cases}
$$
\end{lemma}

\begin{proof}
Consider the contour integral
$$
I=\int_{\Gamma_r}\frac{z^k}{m(z)}\,\frac{dz}{2\pi i}
$$
 over the circle of radius $r$,
large enough so that it encloses all the roots $a_1,\dots,a_d$.
By the residue theorem,
$$
I=\sum_{i=1}^d\frac{a_i^k}{m'(a_i)}.
$$
On the other hand,
for large values of $r$,
$$
I\rightarrow\int_{\Gamma_r}\frac{z^k}{z^d}\,\frac{dz}{2\pi i},
$$
so that $I=0$ if $0\leq k\leq d-2$ and $I=1$ for~\mbox{$k=d-1$}.
\end{proof}

\begin{remark}
Lemma~\ref{L: trace=delta} is due to Euler.
See~\cite[p.~90]{Artin} or the proof of~\cite[Theorem III.5.10, p.~97]{Stichtenoth} for a more elementary proof.
\end{remark}

Applying Lemma~\ref{L: trace=delta} to Definition~\eqref{E:chiP},
we may write for $n\geq-1$,
$$
\chi_P(aT^kP^n)=\exp\biggl(\frac{2\pi i}p
\Tr_{\F_q/\F_p}\biggl(\Tr_{\q(P)/\F_q}\left(\frac P{P'}aT^kP^{n}+(P)\right)\biggr)\biggr).
$$
In this formula,
we do not need to assume anymore that $P$ is monic.
Note that it also gives the right value in the case~$P=1/T$ for~$n\geq1$,
by Definition~\eqref{E:chiinfty}.

\begin{remark}
In general,
$\smash{\int_{\Gamma_r}z^km^n(z)\,{dz}/{2\pi i}}$ vanishes for $n\neq-1$,
if~$\Gamma_r$ encircles every root of~$m$.
This follows for $n\leq-2$ by the same limit argument as in the proof of Lemma~\ref{L: residues},
and for $n\geq0$,
it follows since the integrand is holomorphic.
Therefore,
we could write the character symbolically as
$$
\chi_P(\xi)=\exp\biggl(\frac{2\pi i}p\Tr_{\F_q/\F_p}\biggl(\oint_{P}\xi\,\frac{dT}{2\pi i}\biggr)\biggr),
$$
where the notation indicates that all the roots of the polynomial~$P$ are to be encircled.
This motivates the following definition.
\end{remark}

\begin{definition}
For a Laurent series $\xi=x_{n}P^{n}+x_{n+1}P^{n+1}+\dots$,
 with $x_i\in\F_q[T]$ of degree $\deg x_i<\deg P$,
 the {\em sum of the residues of\/ $\xi$ at the points where $P$ vanishes\/} is
\begin{gather*}
\res_P(\xi)=\Tr_{\q(P)/\F_q}(x_{-1}/P'+(P)).
\end{gather*}

And for a Laurent series $\xi=a_{n}T^{-n}+a_{n+1}T^{-n-1}+\dots$ with coefficients $a_i\in\F_q$,
the {\em residue at infinity\/} is
\begin{gather*}
\res_\infty(\xi)=-a_{1}.
\end{gather*}
\end{definition}

With these definitions,
we can simply write,
for an irreducible polynomial~$P$ or~$P=1/T$,
\begin{gather*}
\chi_P(\xi)=\exp\biggl(\frac{2\pi i}p\Tr_{\F_q/\F_p}(\res_P(\xi))\biggr).
\end{gather*}
Note that for~$P\neq1/T$,
$\chi_P$ is trivial on~$\oo_P$.
On the other hand,
$\chi_\infty$ is trivial on the ideal~$T^{-2}\oo_\infty$ of~$\oo_\infty$.\medskip

After these preliminaries for $\q=\F_q(T)$,
it is easy to construct a character on~$K_v^+$.
Recall that we have chosen a function $T$ such that $K$ is a finite separable extension of $\q$.
Define for $\xi\in K_v^+$,
\begin{gather*}
\chi_v(\xi)=\chi_P(\Tr_{v/P}(\xi)),
\end{gather*}
where~$v_P$ is the restriction of~$v$ to~$\q$ and $\Tr_{v/P}$ denotes the trace from $K_v$ to $\q_P$.
Recalling that $\smash{\Tr_{v/P}}$ is an additive continuous map of $K_v$ onto $\q_P$,
we see that $\chi_v$ is a nontrivial character of $K_v^+$.
By Lemma~\ref{L:chars},
we have proved:

\begin{theorem}\label{T: K=K^}
$K_v^+$ is naturally its own character group if we identify the character\/ $\xi\mapsto\chi_v(\eta\xi)$ with the element\/ $\eta\in K_v^+$.
\end{theorem}

The {\em different\/} of $K_v$ over $\q_P$ is defined by
$$
\diff_{v/P}^{-1}=\{\eta\in K_v\colon\Tr_{v/P}(\eta\xi)\in\oo_P\text{ for all }\xi\in\oo_v\}.
$$
The different is clearly an $\oo_v$-module that contains $\oo_v$.
Moreover,
by separability,
it is not all of $K_v$.
Therefore,
$$
\diff_{v/P}=\pi_v^{d(v/P)}\oo_v
$$
 for some exponent $d(v/P)\geq0$.

\begin{lemma}
Let\/~$v$ be a finite valuation.
The character\/ $\xi\mapsto\chi_v(\eta\xi)$ associated with\/~$\eta$ is trivial on\/ $\oo_v$ if and only if\/ $\eta\in\diff_{v/P}^{-1}$.

For an infinite valuation,
that is,
$v(T)<0,$
the character\/ $\xi\mapsto\chi_v(\eta\xi)$ associated with\/~$\eta$ is trivial on\/ $\oo_v$ if and only if\/ $\eta\in T^{-2}\diff_{v/\infty}^{-1}$.
\end{lemma}

\begin{proof}
We only prove this for a valuation above infinity.
Let $\smash{\eta\in T^{-2}\diff_{v/\infty}^{-1}}$.
Then $\smash{T^2\eta}$ lies in the inverse different,
hence for~$\xi\in\oo_v$,
we have $\smash{\Tr_{v/\infty}(T^2\eta\xi)\in\oo_\infty}$.
By linearity of the trace,
$\Tr_{v/\infty}(\eta\xi)\in T^{-2}\oo_\infty$,
so that the character value $\smash{\chi_v(\eta\xi)}$ is $1$.

If $\smash{\eta\not\in T^{-2}\diff_{v/\infty}^{-1}}$,
then we can find~$\xi\in\oo_v$ such that $\Tr_{v/\infty}(T^2\eta\xi)\not\in\oo_\infty$.
This means that~$\Tr_{v/\infty}(\eta\xi)$ has an expansion of the form $a_nT^{-n-2}+a_{n+1}T^{-n-3}+\dots$,
where~$a_n\neq0$ and~$n\leq-1$.
Then~$T^{n+1}\xi\in\oo_v$,
and~$\chi_v(\eta T^{n+1}\xi)=\exp\bigl(-\frac{2\pi i}p a_n\bigr)$ is nontrivial.
\end{proof}

The ramification index of a valuation $v$ that restricts to $v_\infty$ on $\q$ (i.e.,
$v(T)<0$) is denoted by~$e(v/\infty)$ (also see~\eqref{E:e}).
The {\em canonical exponent\/} is defined by
\begin{gather}\label{E: kv}
k_v=\begin{cases}
d(v/P)&\text{if }v(T)\geq0\text{ and $v$ restricts to }v_P,\\
d(v/\infty)-2e(v/\infty)&\text{if }v(T)<0.
\end{cases}
\end{gather}
In particular,
$k_P=0$ for every finite valuation of $\q$,
and $k_\infty=-2$.
The last lemma can be summarized as follows:
the character $\chi_v(\eta\xi)$ is trivial on~\mbox{$\xi\in\oo_v$} if and only if $\eta\in\pi_v^{-k_v}\oo_v$.
In general,
$k_v\geq0$,
but above infinity,
 the canonical exponent may be negative.
Moreover,
it depends on the choice of the function $T$ in $K$.
In Section~\ref{S: RR} we will see that $\sum_vk_vv$ is a canonical divisor of the curve $\Curve$.\medskip

Let $\mu$ be a Haar measure for $K_v^+$.
As in~\cite[Lemmas~2.2.4 and~2.2.5]{Tate},
for a measurable set $M$ in $K_v^+$ and $\alpha\neq0$ in $K_v$,
$$
\mu(\alpha M)=|\alpha|_v\mu(M).
$$
This explains our choice of normalization for the absolute value:
the norm~$|\alpha|_v$ is the factor by which the additive group $K_v^+$ is stretched under the transformation $\xi\mapsto\alpha\xi$.
For the integral,
this means that
$$
\int_{K_v^+} f(\xi)\,\mu(d\xi)=|\alpha|_v\int_{K_v^+} f(\alpha\xi)\,\mu(d\xi).
$$

Let us now select a fixed Haar measure for the additive group $K_v^+$.
Theorem~\ref{T: K=K^} enables us to do this in an invariant way by selecting the measure which is its own Fourier transform under the interpretation of $K_v^+$ as its own character group established in that theorem.
We state the choice of measure which does this,
writing~$d_v\xi$ instead of $\mu(d\xi)$ and $q_v=q^{\deg v}$.
Define
\begin{gather*}
d_v\xi=\text{that measure for which }\oo_v\text{ gets measure }q_v^{-k_v/2}.
\end{gather*}

The Fourier transform is defined for integrable functions,
and the inversion formula holds for integrable continuous functions for which the Fourier transform is also integrable.
See~\cite[Theorem~2.2.2]{Tate} for details.

\begin{theorem}
If we define the Fourier transform $\Fo_vf$ of an integrable continuous function~$f$ by
\begin{gather}\label{E:F_v}
\Fo_v f(\eta)=\int_{K_v^+}f(\xi)\chi_v(\eta\xi)\,d_v\xi,
\end{gather}
then with our choice of measure,
the following inversion formula holds:
$$
f(\xi)=\int_{K_v^+}\Fo_v f(\eta)\chi_v(-\xi\eta)\,d_v\eta=\Fo_v\Fo_v f(-\xi)
$$
\end{theorem}

\begin{example}
For the field $\q=\F_q(T)$ of rational functions,
the indicator function $\pr_P^0$ of the set $\oo_P$ is self-dual for every finite valuation~$v_P$.
Indeed,
for $\xi=\sum_nx_nP^n$,
we have $\chi_P(\xi)=\chi_P(x_{-1}/P)$.
Given~\mbox{$\eta\not\in\oo_P$},
say~\mbox{$\eta=\sum_{n=-N}^\infty y_nP^n$} for some $N>0$ and $y_{-N}\neq0$,
we can find an element~$\smash{\xi=xP^{N-1}\in\oo_P}$ such that $\chi_P(\eta\xi)$ is not trivial.
Therefore,
the integral~\mbox{$\smash{\int_{\oo_P}\chi_P(\eta\xi)\,d_P\xi}$} vanishes for~$\eta\not\in\oo_P$,
and by our choice of Haar measure,
it equals~$1$ for~$\eta\in\oo_P$.
Therefore,
$\smash{\pr_P^0}$ is its own Fourier transform.

For the valuation at infinity,
we find similarly that the indicator function~$\pr_\infty^1$ of the maximal ideal~$T^{-1}\oo_\infty$ is its own Fourier
transform.
\end{example}

\subsubsection{Multiplicative Characters and Measure}

Our first insight into the structure of the multiplicative group $K_v^*$ of $K_v$ is given by the continuous homomorphism $\alpha\mapsto|\alpha|_v$ of $K_v^*$ into the multiplicative group of powers of $q_v=q^{\deg v}$.
The kernel of this homomorphism,
the subgroup of all $\alpha$ with $|\alpha|_v=1$,
 is denoted by $\oo_v^*$.
This group is compact and open.

Concerning the characters of $K_v^*$,
the situation is different from that of~$K_v^+$.
Indeed,
we are interested in all continuous multiplicative maps of $K_v^*$ into the complex numbers,
not only the bounded ones,
and shall call such a map a quasi-character.
We call a quasi-character {\em unramified\/} if it is trivial on $\oo_v^*$.
The reader may find in~\cite[Section~2.3]{Tate} a (partial) description of the characters of~$\oo_v^*$.
We will only be concerned with unramified characters.

\begin{lemma}
The unramified quasi-characters are the maps of the form
$$
|\alpha|_v^s= q_v^{-sv(\alpha)},
$$
where\/ $s$ is determined modulo\/ $2\pi i/\log q_v$.
\end{lemma}

As in \cite[Section 2.3]{Tate},
we will be able to select a Haar measure $d_v^*\alpha$ on $K_v^*$ by relating it to the measure $d_v\xi$ on $K_v^+$.
We will need a multiplicative Haar measure that gives the subgroup $\oo_v^*$ unit measure.\footnote
{This choice is different from Tate's,
who does not include the factor $q_v^{k_v/2}$.}
To this effect we choose as our standard Haar measure on $K_v^*$:
\begin{gather*}
d_v^*\alpha=\frac{q_v^{k_v/2}}{1-q_v^{-1}}\frac{d_v\alpha}{|\alpha|_v}.
\end{gather*}
Since $\oo_v=\oo_v^*\cup\pi_v\oo_v$ is a disjoint union,
we obtain

\begin{lemma}\label{L: vol Ovstar}
The measure\/ $d_v^*\alpha$ gives\/ $\oo_v^*$ unit volume.
\end{lemma}

\subsubsection{The Local Zeta Function}
\label{S: local zeta}

The computations are as in~\cite[Section 2.5, p.~319]{Tate},
the $\pr$-adic case.
We present here only the case when the multiplicative character is unramified.

We use the notation $\pr_v^n$ to denote the indicator function of the set $\pi_v^n\oo_v$.
We have the Fourier transform
\begin{gather}\label{E:Fo_vpr^n}
\Fo_v{\pr_v^n}=q_v^{-n-k_v/2}{\pr_v^{-n-k_v}}.
\end{gather}
The local zeta function is the Mellin transform,
\begin{align}\label{E:zeta_v}
\zeta_v({\pr_v^n},s)=\int_{\pi_v^n\oo_v}|\alpha|_v^s\,d_v^*\alpha
=\sum_{k=n}^\infty q_v^{-ks}=\frac{q_v^{-ns}}{1-q_v^{-s}}.
\end{align}
Its dual is
\begin{align*}
\zeta_v(\Fo_v{\pr_v^n},s)=q_v^{-n-k_v/2}\,\frac{q_v^{(n+k_v)s}}{1-q_v^{-s}}.
\end{align*}
The local zeta function is periodic with period ${2\pi i}/\log q_v$,
hence also with period ${2\pi i}/{\log q}$,
independent of~$v$.

\subsection{The Global Theory}

We refer to \cite[Section 3]{Tate} for the theory of abstract restricted direct products.

\subsubsection{Additive Theory}

Write $\A_K$ for the ring of adeles of $K$ with measure
$$
d\x=\prod_vd_v\x_v.
$$
We define a character on $\A_K$ by
$$
\chi(\x)=\prod_v\chi_v(\x_v)
$$
for an adele $\x=(\x_v)_v$.
Since the different~$\diff_{v/P}$ is nontrivial for only finitely many~$v$,
it follows that $\A_K$ is its own character group via the identification $\x\mapsto\chi(\y\x)$ of an adele $\y$ with a character on the adeles.

The Fourier transform of an integrable function $f$ on the adeles is defined by
$$
\Fo f(\y)=\int_{\A_K}\chi(\y\x)f(\x)d\x.
$$
Note that the Fourier transform is almost an involution:
$\Fo\Fo f(\x)=f(-\x)$ for every continuous integrable function for which the Fourier transform is also integrable.\medskip

We get an additive fundamental domain $\A_K/K$ of unit volume.
Note that~$\A_K/K$ can be identified with a subgroup of $\A_K$,
unlike in the number field case (see also~\cite[Lemma~4.1.4]{Tate}).

\begin{example}
We compute $\A/\q$ for $\q=\F_q(T)$ by a procedure reminiscent of the computation of the partial fraction decomposition (see Lemma~\ref{L: partial fractions} below).
Let~$\x$ be an adele.
For each finite $P$-adic valuation,
substract a rational function of the form $f/P^k$ to cancel the denominator of~$\x$.
Thus we can substract an element of $\q$ so that each finite component of $\x$ becomes a regular function in~$\oo_P$.
Then we can still substract a polynomial from the infinite component of $\x$,
so that this component can be brought into $T^{-1}\oo_\infty$.
We find
$$
\A/\q=T^{-1}\oo_{\infty}\times\prod_{P\neq1/T}\oo_P.
$$
Since $T^{-1}\oo_{\infty}$ and each $\oo_P$ have unit volume,
$\A/\q$ has unit volume.
\end{example}

The character $\chi$ on the adeles has the global property that $\chi(\xi)=1$ for every $\xi\in K$.
This explains our particular choice of local characters.
To show this,
we first need a lemma.

\begin{lemma}\label{L: partial fractions}
A rational function\/ $f$ has a partial fraction decomposition,
\begin{gather}\label{E: partial fractions}
f(T)=\sum_{P\colon v_P(f)<0}\sum_{n=1}^{-v_P(f)}f_{P,n}(T)P^{-n}+f_{\infty}(T),
\end{gather}
where each\/ $f_{P,n}$ is a polynomial in\/ $T$ of degree less than\/ $\deg P,$
 and\/ $f_{\infty}$ is a polynomial in\/ $T$.
\end{lemma}

\begin{proof}
Let $P$ be an irreducible polynomial such that $v_P(f)<0$.
Thus,~$P$ is a factor of the denominator of $f$.
Let $n$ be the exponent of $P$ in the denominator.
Compute the class of $P^nf$ modulo $P$.
Thus we find a polynomial~$f_{P,n}$ of degree less than $\deg P$ such that $f-f_{P,n}P^{-n}$ has the same factors in its
denominator as~$f$,
to the same power,
except that $P$ occurs to a lower power.
We continue until
$$
f-\sum f_{P,n}P^{-n}
$$
 has a trivial denominator,
and hence is a polynomial $f_\infty(T)$.
\end{proof}

\begin{lemma}
For every\/ $\xi\in K$ we have $\chi(\xi)=1$.
\end{lemma}

\begin{proof}
We prove this directly for $K=\q$.
It then follows in general from the definition of $\chi_v$.

Let $\xi\in\q$ be a rational function,
and compute the partial fraction decomposition of $\xi$.
We show that each term in~\eqref{E: partial fractions} contributes trivially to~$\chi(\xi)$.

First,
for an irreducible polynomial $Q$,
$\chi_Q(f_\infty)=1$.
Also $\chi_\infty(f_\infty)=1$ since $f_\infty$ does not have a $1/T$-term.
Hence $\chi(f_\infty)=1$.

Let
$$
f_{P,n}P^{-n}
$$
 be a term in the partial fraction decomposition of $\xi$ and assume~$n\geq2$.
Then $\chi_Q(f_{P,n}P^{-n})=1$ for $Q\neq P$ and for $Q=P$ since~$n\geq2$.
Also \mbox{$\chi_\infty(f_{P,n}P^{-n})=1$},
since the degree of $f_{P,n}P^{-n}$ is at most $\deg P-1-n\deg P\leq-2$.

Finally,
a term $f_{P,1}/P$ does not contribute in $\chi_Q$ for every irreducible polynomial $Q\neq P$.
Assume that~$P$ is monic and write
$$
f_{P,1}(T)=a_0+a_1T+\dots+a_{d-1}T^{d-1},
$$
where~$d=\deg P$ and $a_i\in\F_q$.
The contribution of $\chi_P$ is obtained from the trace over~$\F_p$ of $a_{d-1}$.
Further,
for $i\leq d-2$,
 $\chi_\infty(a_iT^i/P)=1$ because this rational function has degree $\leq-2$.
Also,~\mbox{$a_{d-1}T^{d-1}/P=a_{d-1}/T+O(T^{-2})$},
since~$P$ is monic.
Hence the contribution of $\chi_\infty$ is obtained from the trace over $\F_p$ of $-a_{d-1}$,
which cancels the contribution of $\chi_P$.
It follows that~\mbox{$\chi(\xi)=1$}.
\end{proof}

As in~\cite[Theorem~4.1.4]{Tate},
it follows that the group of characters that are trivial on $K$ is $K^\perp=K$.

\subsubsection{The Theorem of Riemann--Roch}
\label{S: RR}

From the Poisson Summation Formula~\cite[Lemma~4.2.4]{Tate},
we deduce the theorem of Riemann--Roch~\cite[Theorem~4.2.1]{Tate},
\begin{gather*}
\sum_{\xi\in K}f(\xi\gota)=\frac1{|\gota|}\sum_{\xi\in K}\Fo f(\xi/\gota),
\end{gather*}
where the {\em norm\/} of the idele $\gota$ is defined by
\begin{gather}\label{E:norm}
|\gota|=\prod_v|\gota_v|_v.
\end{gather}

\begin{definition}\label{D: E}
We define the {\em average\/} of\/ $f$ by
\begin{gather*}
Ef(\gota)=\sqrt{|\gota|}\sum_{\alpha\in K^*}f(\alpha\gota).
\end{gather*}
\end{definition}

More precisely,
 $E$ is the combination of a restriction (to the ideles) and a trace (the sum over $K^*$) \cite{Connespersonal}.
It is not a true average over $K^*$,
because this set is infinite.
In terms of the average,
the theorem of Riemann--Roch reads
\begin{gather}\label{E: Poisson E}
Ef(\gota)+f(0){|\gota|^{1/2}}=E(\Fo f)(1/\gota)+\Fo f(0)|\gota|^{-1/2}.
\end{gather}

\begin{remark}\label{R:F is involution}
Clearly,
$E(\Fo\Fo f)=Ef$,
so that the Fourier transform induces an involution on the image of the averaging map,
even though it is not itself an involution.
This is because the averages of $f(\x)$ and $f(-\x)$ coincide.
In general,
 the averages of $f(\x)$ and $f(\alpha\x)$ coincide,
for any $\alpha\in K^*$.
Thus the average of~$f$ only depends on the class of the idele $\gota$ in $\A^*/K^*$.
If $f$ only depends on the associated divisor,
then $Ef(\calD)$ depends only on the class of $\calD$ and on its degree.
\end{remark}

A {\em divisor\/} is a formal sum $\calD=\sum_v\calD_vv$,
with $\calD_v\in\Z$.
The divisor is called {\em positive,\/}
$$
\calD\geq0,
$$
if all coefficients $\calD_v$ are positive.
The {\em degree\/} of $\calD$ is
$$
\deg\calD=\sum_v\calD_v\deg v.
$$
An idele $\gota$ corresponds to a divisor,
$$
(\gota)=\sum_vv(\gota_v)v.
$$
The degree of an idele satisfies $|\gota|=q^{-\deg(\gota)}$ by \eqref{E:norm}.
A function $\alpha\in K^*$ gives a {\em principal divisor\/}
$$
(\alpha)=\sum_vv(\alpha)v.
$$

\begin{remark}\label{R:zeros=poles}
See~[\citen{Weil}, Theorem~IV.4.5] or~\cite[Lemma~7]{Iwasawa} for a proof that $|\alpha|=1$.
That is,
a function has as many zeros as poles.
\end{remark}

To see the geometric significance of the theorem of Riemann--Roch,
take
\begin{gather}\label{E:special f}
f(\x)=\prod_v \pr_v^0(\x_v),
\end{gather}
where $\pr_v^0$ denotes the indicator function of $\oo_v$.
Let $\calD=\sum_v\calD_vv$ be a divisor and let $\gota_\calD=\left(\pi_v^{\calD_v}\right)_v$ be a corresponding idele.
Then
\begin{gather}\label{E:count}
f(\xi\gota_\calD)=\prod_v\pr_v^0(\pi_v^{v(\xi)+\calD_v})
=\begin{cases}1&\text{if }(\xi)+\calD\geq0\\0&\text{otherwise.}\end{cases}
\end{gather}
The functions~\mbox{$\xi\in K$} with $(\xi)+\calD\geq0$ form a vector space over $\F_q$,
$$
L(\calD)=\{\xi\in K\colon\xi=0\text{ or }(\xi)+\calD\geq0\}.
$$
It is the vector space of functions having pole divisor bounded by $\calD$.
We denote its dimension over $\F_q$ by $l(\calD)$,
\begin{gather*}
l(\calD)={\F_q}\text{-}\dim L(\calD).
\end{gather*}
By~\eqref{E:count},
\begin{gather}\label{E:sum=l(D)}
\sum_{\xi\in K}f(\xi\gota_\calD)=q^{l(\calD)}.
\end{gather}

Let the {\em canonical divisor\/} of $\Curve$ be the divisor
$$
\calK=\sum_vk_vv.
$$
We define the {\em genus\/} of~$\Curve$ (or of the function field $K$) by
\begin{gather*}
g=1+{\frac12}{\deg\calK}=1+{\frac12}\sum_vk_v\deg v.
\end{gather*}
Hence $\deg\calK=2g-2$.
By Equation~\eqref{E:Fo_vpr^n},
$
\Fo_v \pr_v^0=q_v^{-k_v/2}\pr_v^{-k_v}
$.
Thus we find that
$$
\Fo f(\x)=q^{1-g}\prod_v{\pr_v^{-k_v}}(\x_v).
$$
Therefore
$$
\Fo f(\xi/\gota)=\begin{cases}
q^{1-g}	&\text{if }(\xi)-\calD\geq-\calK\\
0	&\text{otherwise.}
\end{cases}
$$
Hence $\sum_{\xi\in K}\Fo f(\xi/\gota)=q^{1-g}q^{l(\calK-\calD)}$.
Since $1/|\gota|=q^{\deg\calD}$,
we obtain the theorem of Riemann--Roch in the classical formulation:

\begin{theorem}\label{T: RR}
Let\/ $\Curve$ be a curve of genus\/ $g$ with field of constant functions\/ $\F_q$.
Let\/~$\calK$ be a canonical divisor on\/ $\Curve$ and let\/ $\calD$ be a divisor.
The vector space over\/ $\F_q$ of functions\/ $f$ such that\/ $(f)+\calD$ is positive has dimension\/ $l(\calD)$ given by
\begin{gather*}
l(\calD)=\deg \calD+1-g+l(\calK-\calD).
\end{gather*}
\end{theorem}

It follows that $g$ is an integer (hence the degree of $\calK$ is even).
For $\calD=0$ we have $L(0)=\F_q$ and hence $l(0)=1$.
Putting $\calD=0$ in the Riemann--Roch formula,
we find $g=l(\calK)$,
so that $g\geq0$.
See \cite{Artin,Stichtenoth} for the connection with differentials on $\Curve$.

\begin{example}
By the computation of $k_v$ for $\q$,
we see that the canonical divisor of $\P^1$ has degree $-2$,
hence $\P^1$ has genus $0$.
\end{example}

\begin{example}
A curve of genus $1$ always has a rational point.
Indeed,
by Theorem~\ref{T: degree one} below,
we can find a divisor $\calD$ of degree $1$.
By Riemann--Roch,
$l(\calD)=1$,
hence $\calD$ is linearly equivalent (see Definition~\ref{D:lin eq} below) to a positive divisor of degree~$1$,
which is a point on $\Curve$.

Similarly,
a curve of genus $0$ always has a rational point.
\end{example}

\subsubsection{Multiplicative Theory}
\label{S:multiplicative theory}

The multiplicative group of $\A_K$ is $\A^*_K$,
the group of {\em ideles}.
It is the multiplicative group of vectors $(\gota_v)$,
$\gota_v\in K_v^*$,
such that $\gota_v\in\oo_v^*$ for almost all~$v$.
It contains the subgroup
$$
\A^*_0=\ker|\cdot|
$$
of ideles of degree zero.
$\A^*_0$ is a refinement of the group of divisors of degree zero,
and $\A^*_0/K^*$ is a refinement of the ideal class group of $K$.
By Lemma~\ref{L: vol Ovstar},
the measure
$$
d^*\gota=\prod_v d^*_v\gota_v
$$
on $\A^*$ gives the subgroup $\prod_v\oo_v^*$ of $\A_0^*$ unit measure.

\begin{definition}\label{D:lin eq}
Two divisors $\calD$ and $\calD'$ are {\em linearly equivalent\/} if $\calD-\calD'$ is the divisor of a function on\/ $\Curve$.
\end{definition}

Thus $\calD+(\alpha)$ gives all divisors equivalent to $\calD$,
for $\alpha\in K^*$.
By Remark \ref{R:zeros=poles},
 linearly equivalent divisors have the same degree.
We write~\mbox{$\Pic(n)$} for the set of linear equivalence classes of divisors of degree~$n$.
Thus~\mbox{$\Pic(0)$} is a group,
the {\em divisor class group\/} of $\Curve$.
Let $\calP$ be a divisor of degree~$n$.
Then $\calD\mapsto\calP+\calD$ gives a bijection between the classes of degree~$0$ and of degree~$n$.
Thus if $\Pic(n)$ is not empty,
then the number of classes in~$\Pic(n)$ equals that of $\Pic(0)$.
We will see in Theorem~\ref{T: degree one} below that there exist divisors of every degree,
so $\Pic(n)$ is never empty.
We write
\begin{gather*}
h=|\Pic(0)|
\end{gather*}
for the {\em class number\/} of $\Curve$ over $\F_q$.

\begin{example}\label{E:P1}
Let $\Curve$ be a curve of genus $0$ and let $\calD$ be a divisor of degree zero.
By the Riemann--Roch formula,
$l(\calD)\geq1$,
hence there exists a function~$f$ such that $(f)+\calD\geq0$.
Since the degree of this divisor is zero,
we have that~$\calD=(1/f)$.
This shows that $h=1$.

In general,
let $\calD$ be a divisor of degree zero such that $l(\calD)\geq1$.
Then there exists a function $f$ such that $\calD+(f)\geq0$,
so that $\calD$ is linearly equivalent to the trivial divisor.
Then $l(\calD)=1$,
since nonconstant functions always have poles.
On the other hand,
 $l(\calD)=0$ for every nontrivial divisor class of degree zero.
\end{example}

\begin{theorem}
The number of linear equivalence classes in each degree is finite.
\end{theorem}

\begin{proof}
For $\deg\calD\geq g$ we have $l(\calD)\geq1+l(\calK-\calD)\geq1$.
Hence there exists a function $\alpha$ such that $\calD+(\alpha)\geq0$.
It follows that the class of a divisor of degree~$\geq g$ is represented by a positive divisor.
Since there are only finitely many valuations of each degree,
the number of elements of $\Pic(n)$ is finite for $n\geq g$.
But $|\Pic(0)|=|\Pic(n)|$ if $\Pic(n)$ is not empty,
so $h$ is finite.
\end{proof}

The volume of the group of idele classes $\A^*_0/K^*$ is found as follows:
Choose ideles $\cid_1,\dots,\cid_h$ of degree zero representing each class.
Given an idele of degree $0$,
we can first divide by a $\cid_i$ to make its class trivial.
Then we can divide by a function to make it a unit everywhere,
and finally,
we can divide by a constant function in~$\F_q^*$ to make the value of this unit $1$ at one fixed valuation of degree $1$,
or,
if such a valuation does not exist,
we can normalize the idele in a similar manner using $\F_q^*$.
The group~$\prod_v\oo_v^*$ has unit volume.
Also,
$\F_q^*$ contains $q-1$ elements.
Hence the volume of~$\A^*_0/K^*$,
denoted $\kappa$,
 is given by
\begin{gather}\label{E: kappa}
\kappa=\vol(\A^*_0/K^*)=\frac{h}{q-1}.
\end{gather}
Unlike the number field case,
$\A^*_0/K^*$ can be identified with a subgroup of~$\A^*_0$,
even if there is more than one valuation above infinity.
(See also~\cite[Section~4.3]{Tate}.)

\subsubsection{The Zeta Function}

As in \cite{Tate},
for a function $f$ on the adeles such that $f$ and $\Fo f$ are of fast decay as $|\gota|\to\infty$,
the zeta function is defined for $\Re s>1$ by
\begin{gather}\label{E: zeta_C=int}
\zeta_\Curve(f,s)=\int_{\A^*}f(\gota)|\gota|^s\,d^*\gota.
\end{gather}

This definition is not natural,
since we integrate the additive function $f$ over the multiplicative group $\A^*$.
However,
by summing over $K^*$,
we obtain in terms of the average,\footnote
{This formula is the analogue of the expression
$$
\zeta_\Z(s)=\int_0^\infty(\theta(t)-1)t^s\,\frac{dt}t
$$
for the Riemann zeta function,
where $\theta(t)$ is defined after \eqref{E: theta flip}.}
\begin{gather}\label{E: zeta_C=E}
\zeta_\Curve(f,s)=\int_{\A^*/K^*}Ef(\gota)|\gota|^{s-1/2}\,d^*\gota.
\end{gather}
In this sense,
$\zeta_\Curve(f,s+1/2)$ is the Mellin transform of $Ef$,
completely naturally.

In case $f=\prod_vf_v$ is a product of local functions,
we find the Euler product of~$\zeta_\Curve$ by writing the integral over the ideles as a product of local integrals,
\begin{gather}\label{E:Euler product}
\zeta_\Curve(f,s)=\prod_v\zeta_v(f_v,s).
\end{gather}
The local factors have been computed in Section~\ref{S: local zeta} for $f_v=\pr_v^n$.
If $f(\gota)$ only depends on the divisor associated with $\gota$,
then we obtain,
 by summing over cosets of~$\A^*/\prod_v\oo_v^*$,
\begin{gather}\label{E:zeta sum D}
\zeta_\Curve(f,s)=\sum_{\calD}f(\calD)q^{-s\deg\calD},
\end{gather}
where $f(\calD)=f((\pi_v^{\calD_v})_v)$.\medskip

We define the {\em zeta function\/} of $\Curve$ by
$$
\zeta_\Curve(s)=q^{(g-1)s}\zeta_\Curve(f,s),
$$
for the special choice\footnote
{By Theorem~\ref{T: degree one},
we could also take
$
f=\prod_{v}\pr_v^{a_v},
$
where the exponents $a_v$ are chosen so that $\sum_va_v\deg v=1-g$,
to define $\zeta_\Curve(s)=\zeta_\Curve(f,s)$,
without the need of the factor $q^{(g-1)s}$.} $f=\prod_v\pr_v^0$ as in \eqref{E:special f}.
The factor $q^{(g-1)s}$ has been inserted so that the functional equation is self-dual,
see Section~\ref{S:two-variable}.
By~\eqref{E:zeta sum D},
this function can be written for $\Re s>1$ as a sum over positive divisors,
$$
\zeta_\Curve(s)=q^{(g-1)s}\sum_{\calD\geq0}q^{-s\deg\calD}.
$$
This is the analogue of~\eqref{D:Rzeta}.
By~\eqref{E:Euler product},
we obtain the Euler product over all valuations,
 analogous to \eqref{E: Euler prod},
\begin{gather}\label{E: Euler zetaC}
\zeta_\Curve(s)=q^{(g-1)s}\prod_v\frac1{1-q_v^{-s}}.
\end{gather}

To compute this function,
we use~\eqref{E: zeta_C=E}.
By~\eqref{E:sum=l(D)},
we obtain
\begin{gather}\label{E:zetaC Cl}
\zeta_\Curve(s)=q^{(g-1)s}\sum_{n=0}^\infty q^{-ns}\sum_{\calD\in\Pic(n)}\frac{q^{l(\calD)}-1}{q-1}.
\end{gather}
Let $\mu$ be the minimal positive degree of a divisor.
Then $\Pic(n)$ is empty if $n$ is not a multiple of $\mu$,
hence \eqref{E:zetaC Cl} only contains terms with $n$ a multiple of $\mu$.
For a divisor $\calD$ of degree $n\geq2g-1$,
$l(\calD)=n+1-g$.
Hence the infinite series \eqref{E:zetaC Cl} becomes geometric.
We thus find that $\zeta_\Curve$ has simple poles at the values for $s$ such that $q^{\mu s}=q^\mu$ and at values such that $q^{\mu s}=1$.
To finish the computation,
we first show that there exists a divisor of degree~$1$ (see \cite{Stichtenoth} or \cite[V.5.3]{Eichler}).

\subsubsection*{Constant Field Extensions}

Recall that $\F_q$ is the field of constants of $K$,
i.e.,
$\F_q$ is algebraically closed in $K$.
We denote the constant field extension of degree $n$ by
$$
K_n=K[X]/(m),
$$
where~$m$ is an irreducible polynomial over $\F_q$ of degree $n$.
The corresponding zeta function (without the $q^{(g-1)s}$ factor) is defined by its Euler product over all valuations $w$ of $K_n$,
\begin{gather}\label{E: zeta q^n}
\zeta_{\Curve/\F_{q^n}}(s)=\prod_w\frac1{1-q_w^{-s}}.
\end{gather}

The Frobenius automorphism of $\F_{q^n}$ is $\phi_q^n$,
hence the Frobenius flow of $\smash{\Curve/\F_{q^n}}$ is generated by $\phi_q^n$.
Let $w$ be a valuation of $K_n$.
Its restriction to~$K$,
say $v$,
 corresponds to an orbit of $\phi_q$ of $\deg v$ points in $\smash{\Curve(\Fpa)}$.
Under $\phi_q^n$,
this orbit splits into $d$ orbits,
each of $\deg (v)/d$ points,
where
$$
d=\gcd(n,\deg v).
$$
Hence $\deg w=\deg(v)/d$,
i.e.,
 $K_n(w)$ has degree $\deg(v)/d$ over $\F_{q^n}$,
and hence degree $n/d$ over $K(v)$.
This means that
$$
q_w=q_v^{n/d}.
$$
Moreover,
there are $d$ different valuations of $K_n$ that restrict to $v$.
We find that the local factor corresponding to $v$ in the Euler product for $\zeta_{\Curve/\F_{q^n}}(s)$ is given by
$$
\prod_{w|v}\frac1{1-q_w^{-s}}=\frac1{\left(1-q_v^{-ns/d}\right)^{d}}.
$$

Now $\deg(v)/d$ is relatively prime to $n/d$.
Hence $e^{2\pi i\deg(v)/n}$ is a primitive $n/d$-th root of unity.
We deduce that
$$
\bigl(1-X^{n/d}\bigr)^d=\prod_{k=1}^n\bigl(1-e^{2\pi ik\deg (v)/n}X\bigr).
$$
For $X=q_v^{-s}$,
we obtain the following lemma.

\begin{lemma}\label{L: zeta q^n}
Let\/ $\zeta_{\Curve/\F_{q^n}}(s)$ be the zeta function of\/ $K_n,$
as in~\eqref{E: zeta q^n}.
Then
$$
\zeta_{\Curve/\F_{q^n}}(s)=\prod_{k=1}^n\zeta_{\Curve/\F_q}\biggl(s+k\frac{2\pi i}{n\log q}\biggr).
$$
\end{lemma}

Consider $\Curve$ again over $\F_q$,
and let $\mu$ be the minimal positive degree of a divisor.
In particular,
the degree of every valuation is a multiple of~$\mu$.
Hence,
by~\eqref{E: zeta q^n},
$\zeta_{\Curve/\F_q}(s)$ has period $2\pi i\,/\,\mu\log q$.
By Lemma~\ref{L: zeta q^n},
we find that $\zeta_{\Curve/\F_{q^\mu}}(s)$ has a pole of order~$\mu$ at $s=1$.
Since this pole is simple,
we find that~\mbox{$\mu=1$}.
We have proven the following theorem.

\begin{theorem}\label{T: degree one}
There exists a divisor on\/ $\Curve$ of degree $1$.
\end{theorem}

\begin{example}
The curve $X^2+Y^2+1=0$ has no point over the rational numbers,
and every divisor consists of at least two points.
Indeed,
every divisor defined over the rational numbers or the real numbers has an even degree.
By the last theorem,
we cannot have a similar situation over a finite field.
\end{example}

We finish the computation of $\zeta_\Curve$.
For $\Curve=\P^1$,
i.e.,
genus zero,
$l(\calD)=n+1$ for $n=\deg\calD\geq0$.
By \eqref{E:zetaC Cl},
we find
$$
\zeta_{\P^1}(s)=\frac1{(q^s-1)(1-q^{1-s})}.
$$

For higher genus,
we use that
$$
l(\calD)\geq\max\{0,\deg\calD+1-g\},
$$
with equality for $\deg\calD<0$ or $\deg\calD>2g-2$.
We then find
\begin{gather}\label{E: zeta_C}
\zeta_\Curve(s)
=q^{(g-1)s}\sum_{n=0}^{2g-2}q^{-ns}\sum_{\calD\in\Pic(n)}\frac{q^{l(\calD)}-q^{\max\{0,n+1-g\}}}{q-1}+h\zeta_{\P^1}(s),
\end{gather}
where $h$ is the class number of $\Curve$.
We thus obtain
\begin{gather}\label{E: zetaC L}
\zeta_\Curve(s)=q^{gs}L_\Curve(q^{-s})\zeta_{\P^1}(s),
\end{gather}
where $L_\Curve(X)$,
for $X=q^{-s}$,
is given by
$$
L_\Curve(X)=(1-X)(1-qX)\sum_{n=0}^{2g-2}X^{n}\sum_{\calD\in\Pic(n)}\frac{q^{l(\calD)}-q^{\max\{0,n+1-g\}}}{q-1}+hX^g.
$$
It follows that $L_\Curve$ is a polynomial of degree~$2g$.
Moreover,
$L_\Curve(0)=1$.
We write $\omega_\nu$ for the reciprocal zeros of $L_\Curve$,
so that
we have
\begin{gather}\label{E:L=omega}
L_\Curve(X)=\prod_{\nu=1}^{2g}(1-\omega_\nu X).
\end{gather}

\subsubsection{The Two-Variable Zeta Function}
\label{S:two-variable}

In~\cite{Pell} (see also \cite{Expositiones}),
Pellikaan defines the following two-variable zeta function,
\begin{gather*}
\zeta_\Curve(s,t)=\frac{q^{(g-1)s}}{q^t-1}\sum_{n=-\infty}^\infty q^{-ns}\sum_{\calD\in\Pic(n)}q^{tl(\calD)}.
\end{gather*}
This series is convergent for~\mbox{$\Re t<\Re s<0$}.
Using
$$
n+1-g=l(\calD)-l(\calK-\calD)
$$
 for $\deg\calD=n$,
we obtain
$$
\zeta_\Curve(s,t)=\frac{1}{q^t-1}\sum_{n=-\infty}^\infty\sum_{\calD\in\Pic(n)}q^{(t-s)l(\calD)+sl(\calK-\calD)}.
$$
Replacing $s$ by $t-s$ and summing over $\calK-\calD$ instead of $\calD$,
we see that the two-variable zeta function satisfies the functional equation
$$
\zeta_\Curve(s,t)=\zeta_\Curve(t-s,t).
$$

For $\Curve=\P^1$,
we compute this function by using that~\mbox{$l(\calD)=0$} for $\deg\calD<0$ and $l(\calD)=\deg\calD+1$ for~\mbox{$\deg\calD\geq0$}.
Also,
 $h=1$ by Example \ref{E:P1}.
We find
$$
\zeta_{\P^1}(s,t)=\frac1{(q^s-1)(1-q^{t-s})}.
$$

In general,
we have
\begin{gather*}
\zeta_\Curve(s,t)=q^{(g-1)s}\sum_{n=0}^{2g-2}q^{-ns}
\sum_{\calD\in\Pic(n)}\frac{q^{tl(\calD)}-q^{t\max\{0,n+1-g\}}}{q^t-1}+h\zeta_{\P^1}(s,t),
\end{gather*}
where $h$ is the class number of $\Curve$.
By~\eqref{E: zeta_C},
it follows that $\zeta_\Curve(s,1)=\zeta_\Curve(s)$.
This proves the functional equation for the zeta function of $\Curve$,
$$
\zeta_\Curve(1-s)=\zeta_\Curve(s).
$$

\begin{remark}\label{R:symmetry}
Clearly,
by the functional equation,
if $\zeta_\Curve$ does not satisfy the Riemann hypothesis,
it will have zeros on both sides of the line $\Re s=1/2$.
Equivalently,
if not all the numbers $\omega_\nu$ have absolute value $\sqrt{q}$,
then some of them will be larger in absolute value,
and some will be smaller.
\end{remark}

\begin{remark}
See \cite{Expositiones} for more information about the two-variable zeta function.
In particular,
this function satisfies an expression analogous to \eqref{E: zetaC L}.
Remarkably,
by \cite[Example 4.4]{Pell},
it is not true that $\zeta_\Curve(s)$ determines $\zeta_\Curve(s,t)$.
\end{remark}

\section{A Proof of the Riemann Hypothesis for $\Curve$}
\label{S:RH}

\subsection{Points on $\Curve$}
\label{S:points}

Every valuation of $K$ corresponds to $\deg v$ points on $\Curve$ defined over the finite field~$K(v)$.
Let $N_\Curve(n)$ be the number of points on $\Curve$ with values in~$\F_{q^n}$.
Since $K(v)$ is a subfield of $\F_{q^n}$ if and only if $\deg v\mid n$,
we find that
\begin{gather}\label{E:N_C}
N_\Curve(n)=\sum_{\deg v|n}\deg v.
\end{gather}
By~\eqref{E: Euler zetaC},
we have
$$
-\frac1{\log q}\frac{\zeta_\Curve'}{\zeta_\Curve}(s)
=1-g+\sum_{n=1}^\infty q^{-ns}N_\Curve(n).
$$
On the other hand,
taking the logarithmic derivative of the expression~\eqref{E: zetaC L} of~$\zeta_\Curve$ as a rational function of $q^{-s}$,
we find by \eqref{E:L=omega}
$$
-\frac1{\log q}\frac{\zeta_\Curve'}{\zeta_\Curve}(s)
=1-g+\sum_{n=1}^\infty q^{-ns}\left(q^n-\omega_1^n-\dots-\omega_{2g}^n+1\right).
$$
Comparing coefficients,
we obtain
\begin{gather}\label{E:NCn}
N_\Curve(n)=q^n-\omega_1^n-\dots-\omega_{2g}^n+1.
\end{gather}

Recall that the Riemann hypothesis for $\Curve$ can be formulated as $|\omega_\nu|\leq\sqrt q$ for $\nu=1,\dots,2g$.
It follows from the Riemann hypothesis that
$$
\left|N_\Curve(n)-q^n-1\right|\leq2gq^{n/2}.
$$
Conversely,
we have the following lemma,
which states in particular that it suffices to prove~\eqref{E: RH} for all even $n$.

\begin{lemma}\label{L: RH}
If for every $\varepsilon>0$ there exists a natural number\/ $m$ such that the inequality
\begin{gather}\label{E: RH}
\left|N_\Curve(nm)-q^{nm}-1\right|\leq Bq^{nm(1/2+\varepsilon)}
\end{gather}
is satisfied for every\/ $n,$
then the Riemann hypothesis holds for\/~$\zeta_\Curve$.
\end{lemma}

\begin{proof}
Let $\varepsilon>0$.
By Diophantine approximation,
we can find infinitely many $n$ such that $\Re\omega_\nu^{nm}\geq\frac12|\omega_\nu|^{nm}$ for $\nu=1,\dots,2g$.
Then
$$
\left|N_\Curve(nm)-q^{nm}-1\right|\geq\frac12\max_\nu|\omega_\nu|^{nm}.
$$
Letting $n\to\infty$,
we find that $|\omega_\nu|\leq q^{1/2+\varepsilon}$ for every $\nu$.
Since this holds for every $\varepsilon>0$,
we obtain $|\omega_\nu|\leq q^{1/2}$.
By Remark \ref{R:symmetry},
we conclude that $|\omega_\nu|=q^{1/2}$ for every $\nu$.
\end{proof}

\subsection{The Graph of the Frobenius Flow}
\label{S:graph of F}

Figure~\ref{F: Frobenius} depicts the graph
$$
Y=\phi_q(X)
$$
of one step of the Frobenius flow in $\Curve\times\Curve$.
The intersection with the diagonal
$$
\Delta\colon Y=X
$$
gives the points $(x,x)$ with $\phi_q(x)=x$.
These are the points on $\Curve$ defined over~$\F_q$,
and their number is $N_\Curve(1)$.
We assume that there is at least one point on the intersection,
which we denote by $(\infty,\infty)$.
We write $v_\infty$ for the corresponding valuation of degree $1$ of $K$.

\begin{figure}[tb]
\begin{picture}(324,240)(-30,-8)
\put(0,0){\epsfxsize 8cm\epsfbox{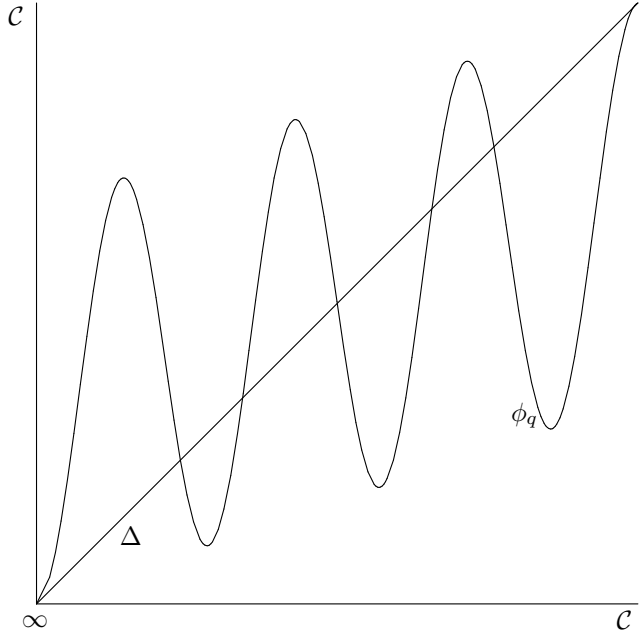}}
\put(0,-6){\makebox(0,0){$\infty$}}
\put(220,-9){$\Curve$}\put(-10,220){$\Curve$}
\put(32,23){$\Delta$}\put(180,70){$\phi_q$}
\end{picture}
\caption[The graph of Frobenius intersected with the diagonal]{The graph of Frobenius intersected with the diagonal.}
\label{F: Frobenius}
\end{figure}

\begin{remark}
The Frobenius automorphism is smooth,
since it is a polynomial map.
Also,
its derivative vanishes,
so our intuition says that this map should be constant,
or at least locally constant.
Being a polynomial of degree $q$,
it seems to be a $q$-to-one map,
but in fact,
it is one-to-one.
Figure~\ref{F: Frobenius} emphasizes the smoothness and ignores the injectivity of the Frobenius flow.
\end{remark}

The functions defined over $\F_q$ with a pole of order at most $m$ at $\infty$ and no other poles form an $\F_q$-vector space $L_m=L(m(\infty))$.
By Theorem~\ref{T: RR},
the dimension of $L_m$ is
$$
l_m=l(m(\infty))=m+1-g+l(\calK-m(\infty)).
$$
We find
\begin{gather}\label{E:lm<}
m+1-g\leq l_m\leq m+1,
\end{gather}
and
\begin{gather}\label{E:lm=}
l_m=m+1-g\text{ for }m>2g-2.
\end{gather}

Clearly,
$L_{m+1}$ contains $L_m$ as a subspace.
Also,
$l_{m+1}\leq l_m+1$,
since for two functions $f$ and $g$ in $L_{m+1}$ for which $f\not\in L_m$,
we can find a constant $\lambda\in\F_q$ such that $g-\lambda f\in L_m$.
Let
$$
s_1,\dots,s_{l_m}
$$
be a basis for $L_m$ such that $v_\infty(s_{i+1})<v_\infty(s_i)$,
i.e.,
the order of the pole of~$s_i$ at~$\infty$ increases with $i$.
Given $k\geq0$,
we choose coefficients $a_i\in L_k$ to form
\begin{gather*}
f(X,Y)=\sum_{i=1}^{l_m}a_i(X)s_i(Y).
\end{gather*}

\begin{remark}\label{R:arith geo}
For $\Curve=\P^1$,
$s_1(Y)=1$,
$s_2(Y)=Y$,
$s_3(Y)=Y^2,\dots$,
and $f(X,Y)$ is a polynomial in $Y$ with coefficients in $X$,
which we will assume to be of low degree in $X$.
This is analogous to a polynomial in $Y$ with integer coefficients.
Therefore,
 we call $Y$ the `geometric' coordinate,
and $X$ the `arithmetic' coordinate.
\end{remark}

We thus obtain a space of functions $f(X,Y)$ on $\Curve\times\Curve$ defined over~$\F_q$.
The restriction of $f(X,Y)$ to the graph of Frobenius is
\begin{gather*}
f_{|\phi}(X)=f(X,\phi_q(X)).
\end{gather*}
The map $f\mapsto f_{|\phi}$ is $\F_q$-linear.
Note that $s_i(\phi_q(X))=s_i^q(X)$,
since $s_i$ is defined over $\F_q$.
Hence $f_{|\phi}\in L_{k+qm}$.

\begin{lemma}\label{L: f|phi}
For\/ $k<q,$
the map\/ $f\mapsto f_{|\phi}$ is injective,
and hence an isomorphism onto its image.
\end{lemma}

\begin{proof}
Assume $f\mapsto f_{|\phi}=0$.
That is,
$\sum_{i=1}^{l_m}a_is_i^q=0$.
Consider the order of the pole at $\infty$.
If $a_i\neq0$,
then $v_\infty(a_is_i^q)\leq qv_\infty(s_i)\leq -q+qv_\infty(s_j)$ for every~$j<i$.
Further,
$v_\infty(a_js_j^q)\geq -k+qv_\infty(s_j)>-q+qv_\infty(s_j)$.
Hence the pole of the nonzero term of highest order in~$\smash{f_{|\phi}}$ is not cancelled by the pole of any of the other terms.
It follows that the highest nonzero term vanishes.
We conclude that there is no highest nonzero term,
and hence~$f=0$.
\end{proof}

We take the coefficients~$a_i$ to be $p^\mu$-th powers,
for $p^\mu<q$,
so that $f_{|\phi}$ is a $p^\mu$-th power.
Hence the coefficients are of the form
$$
a_i=b_i^{p^\mu}.
$$
We choose $b_i\in L_n$,
so that $a_i\in L_{p^\mu n}$.
To be able to apply Lemma~\ref{L: f|phi},
we assume that
\begin{gather}\label{pl}
p^\mu n<q.
\end{gather}\medskip

We can also restrict $f$ to the diagonal:
\begin{gather*}
f_{|\Delta}(X)=f(X,X).
\end{gather*}
We obtain the two restriction maps,
for $k<q$,
\begin{gather*}\begin{CD}
f	@>>>	f_{|\Delta}\\
@|\\
f_{|\phi}
\end{CD}\end{gather*}
where the vertical equality is the isomorphism of Lemma~\ref{L: f|phi}.

Since $f_{|\phi}$ only has a pole at $\infty$,
 of order at most $p^\mu n+qm$,
it also has at most this many zeros,
counted with multiplicity.
If therefore $f_{|\Delta}=0$ and $f_{|\phi}\neq0$,
then we have a function $f_{|\phi}$ on $\Curve$,
obtained by restricting a function on~\mbox{$\Curve\times\Curve$} that vanishes on the diagonal,
except at $(\infty,\infty)$.
Apart from~$(\infty,\infty)$,
the diagonal intersects the graph of Frobenius at~$N_\Curve(1)-1$ points.
Since~$f_{|\phi}$ is a $p^\mu$-th power,
this function has at least~\mbox{$p^\mu(N_\Curve(1)-1)$} zeros,
counted with multiplicity.
Comparing with the pole of $f_{|\phi}$,
 we find the inequality~\mbox{$N_\Curve(1)\leq1+n+\frac q{p^\mu}m$}.
By~\eqref{pl},
$1+n\leq qp^{-\mu}$,
hence we obtain
\begin{gather}\label{E: N1m}
N_\Curve(1)\leq\frac q{p^\mu}(m+1).
\end{gather}
In particular,
 for given $\mu$,
the best bound for $N_\Curve(1)$ is obtained when $m$ is as small as possible.

\begin{example}
For genus zero,
we take
$$
f(X,Y)=X^{p^\mu}-Y^{p^\mu}.
$$
Then $f_{|\phi}(X)=X^{p^\mu}-X^{qp^\mu}$ and $f_{|\Delta}(X)=0$.
The function $f_{|\phi}$ has a pole at infinity of order $qp^\mu$,
and at least $p^\mu(N_\Curve(1)-1)$ zeros,
counted with multiplicity.
Hence the number of points on the projective line over $\F_q$ satisfies $N_{\P^1}(1)\leq q+1$.
In fact,
equality holds.
\end{example}

For higher genus,
we will not explicitly construct a function $f$ such that \mbox{$f_{|\Delta}=0$} and $f_{|\phi}\ne0$,
but we prove that such a function exists.
The space of functions~$f(X,Y)$ that we have constructed has dimension $l_nl_m$.
Assume that
$$
n,m\geq g.
$$
Then $l_nl_m\geq(n+1-g)(m+1-g)$ by~\eqref{E:lm<}.
The functions~$f_{|\Delta}$ lie in $L_{p^\mu n+m}$.
To assure the existence of a nontrivial function such that $f_{|\Delta}=0$,
we choose~$n$ and $m$ so that~\mbox{$(n+1-g)(m+1-g)>l_{p^\mu n+m}$},
since then the kernel of the map $f\mapsto f_{|\Delta}$ is nontrivial.
Since~\mbox{$p^\mu n+m>2g-2$} and by~\eqref{E:lm=},
this means that we want
$$
(n+1-g)(m+1-g)>p^\mu n+m+1-g,
$$
or equivalently,
\begin{gather}\label{ker}
(m+1-g-p^\mu)(n-g)>p^\mu g.
\end{gather}
Since we want to choose $m$ as small as possible,
we choose $n$ as large as possible.
The largest value for $n$ so that~\eqref{pl} is satisfied is
$$
n= q{p^{-\mu}}-1.
$$
We thus obtain from \eqref{ker} a lower bound for $m$,
which we can write as
\begin{gather*}
\frac q{p^\mu}(m+1)>q+g\biggl(\frac q{p^\mu}+\frac{p^\mu}{1-(g+1)p^\mu/q}\biggr)>q+g\biggl(\frac q{p^\mu}+p^\mu\biggr).
\end{gather*}
Thus the upper bound for $N_\Curve(1)$ that we can derive from~\eqref{E: N1m} is best possible if~\mbox{$q/p^\mu=p^{\mu}$}.
Therefore,
we assume that $q$ is an even power of~$p$,
as we may by Lemma~\ref{L: RH},
and we choose $\mu$ such that $p^\mu=\sqrt q$.
Then \eqref{ker} is equivalent to
$$
m+1>\sqrt q+2g+\frac{g(g+1)}{\sqrt q-(g+1)}.
$$
For $q>(g+1)^4$,
 this inequality is satisfied for
$$
m+1=\sqrt q+2g+1.
$$

With these choices,
there exists a nontrivial function $f$ on $\Curve\times\Curve$ such that $f_{|\Delta}=0$.
By Lemma~\ref{L: f|phi},
 also $f_{|\phi}$ is nontrivial,
hence we obtain the following theorem.

\begin{theorem}\label{T: upper bound}
For\/ $q>(g+1)^4,$
a square,
we have
$$
N_\Curve(1)=|\Curve(\F_q)|\leq q+(2g+1)\sqrt{q},
$$
where\/ $g$ is the genus of\/ $\Curve$.
\end{theorem}

Note that the above argument depends on the existence of a point $\infty$ on~$\Curve(\F_q)$.
If such a point does not exist,
then $N_\Curve(1)=0$ and the inequality for~$N_\Curve(1)$ is trivially satisfied.

\subsection{Galois Covers of $\Curve$}
\label{S:Galois}

\begin{figure}[tb]
\begin{picture}(324,240)(-30,-8)
\put(0,0){\epsfxsize 8cm\epsfxsize 8cm\epsfbox{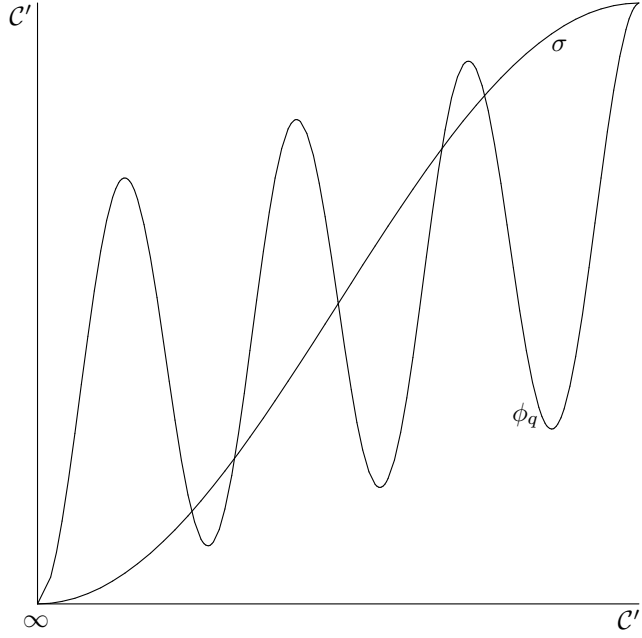}}
\put(0,-6){\makebox(0,0){$\infty$}}
\put(220,-9){$\Curve'$}\put(-10,220){$\Curve'$}
\put(195,210){$\sigma$}\put(180,70){$\phi_q$}
\end{picture}
\caption[The graph of Frobenius intersected with the graph of $\sigma$]
{The graph of Frobenius intersected with the graph of $\sigma$.}
\label{F: sigma}
\end{figure}

Let $\Curve'\longrightarrow\Curve$ be a Galois cover of curves,
i.e.,
the function field $L$ of $\Curve'$ is a Galois extension of $K$.
The field $L$ can be written as
$$
K[X]/(m(X)),
$$
for some irreducible polynomial $m$.
For every automorphism $\sigma$ of $L$ over $K$,
the element $\sigma(X)$ lies in $L$,
hence we can find a polynomial $f$ with coefficients in $K$ such that $\sigma (X)=f(X)+(m)$.
It follows that the action of $\sigma$ on $\Curve'$ is algebraic,
induced by $X\mapsto f(X)$.
(See Figure~\ref{F: sigma}.)\medskip

Let $w$ be a valuation of $L$,
and $v$ its restriction to $K$.
The {\em decomposition group\/} of $w$ over $v$,
$$
Z(w/v)=\{\sigma\in\Gal(L/K)\colon w(\sigma x)>0\text{ if }w(x)>0\}
$$
is the group of continuous automorphisms of $L$ over $K$,
and the {\em ramification group\/} of $w$ over $v$ is its subgroup
$$
T(w/v)=\{\sigma\in Z(w/v)\colon w(x-\sigma x)>0\text{ if }w(x)\geq0\},
$$
the group of automorphisms that act trivially modulo the maximal ideal $\pi_w\oo_w$.
We denote the order of $T(w/v)$ by $e(w/v)$,
\begin{gather}\label{E:e}
e(w/v)=|T(w/v)|,
\end{gather}
the {\em order of ramification\/} of $w$ over $v$.
The group $Z(w/v)/T(w/v)$ is isomorphic to the Galois group of $L(w)$ over $K(v)$.
We denote its order by $f(w/v)$,
the {\em degree of inertia\/} of $w$ over $v$,
so that $\deg w=f(w/v)\deg v$.
It is generated by the Frobenius automorphism of $L(w)$ over $K(v)$,
the automorphism that raises an element to the power $q^{\deg v}$.
There are $e(w/v)$ different automorphisms in $\Gal(L/K)$ that induce this automorphism modulo $\pi_w\oo_w$.

\subsection{Frobenius as Symmetries of a Cover}
\label{S: point frobenius}

We also need a generalization of Theorem~\ref{T: upper bound} to Galois covers.
Let
\begin{gather}\label{E:cover}
\Curve'\longrightarrow\Curve\longrightarrow\P^1
\end{gather}
be the Galois cover corresponding to the Galois closure of $K$ over $\q$.
Let~$G$ be the Galois group of the cover~\mbox{$\Curve'\rightarrow\P^1$}.
For $\sigma\in G$,
we define
\begin{align*}
N_{\Curve'}(1,\sigma)
&=\bigl|\bigl\{x\in\Curve'(\Fpa)\colon x\text{ projects to $\P^1(\F_q)$ and }\phi_{q}(x)=\sigma(x)\bigr\}\bigr|.
\end{align*}

\begin{theorem}\label{T: upper bound, covers}
For\/ $q>(g+1)^4,$
a square,
we have
$$
N_{\Curve'}(1,\sigma)\leq q+(2g'+1)\sqrt{q},
$$
where\/ $g'$ is the genus of\/ $\Curve'$.
\end{theorem}

\begin{proof}
Let $X$ and $Y$ denote the `arithmetic' and `geometric' coordinates on~$\Curve'\times\Curve'$;
see Figure~\ref{F: sigma}.
As in Section~\ref{S:graph of F},
we have the restrictions
\begin{gather*}\begin{CD}
f	@>>>	f_{|\sigma}\\
@|\\
f_{|\phi}
\end{CD}\end{gather*}
where $f_{|\sigma}(X)=f(X,\sigma (X))$ is the restriction of $f(X,Y)$ to the graph of~$\sigma$ and,
as before,
 $f_{|\phi}(X)$ is the restriction to the graph of the Frobenius flow.
Clearly,
if~$f_{|\sigma}$ vanishes,
then $f_{|\phi}$ vanishes at the points counted in $N_{\Curve'}(1,\sigma)$.
The rest of the argument is as before,
applied to~$\Curve'$ and the homomorphism~\mbox{$f_{|\phi}\mapsto f_{|\sigma}$}.
\end{proof}

\begin{theorem}
The curve\/ $\Curve$ satisfies the Riemann hypothesis.
That is,
if\/ $\zeta_\Curve(s)=0$ then\/ $\Re s=1/2$.
\end{theorem}

\begin{proof}
As in~\eqref{E:cover},
let $\Curve'$ be the Galois closure of the cover~\mbox{$\Curve\rightarrow\P^1$},
with Galois group $G$.
Consider the sum
$$
\sum_{\sigma\in G}N_{\Curve'}(1,\sigma).
$$
Above every point $t$ of~$\P^1(\F_q)$,
we have $|G|/e$ points of~\mbox{$\Curve'(\Fpa)$},
where~$e$ is the ramification index of any of the associated valuations in~$\Curve'$.
Further,
for a point~$t'$ of $\Curve'$ above~$t$,
we have~$e$ different automorphisms in $G$ that induce Frobenius on the residue class field.
Hence in the sum,
each point of~$\P^1(\F_q)$ is counted~$|G|$ times.
Since~\mbox{$\P^1(\F_q)$} has~\mbox{$q+1$} points,
we obtain
\begin{gather*}
\sum_{\sigma\in G}N_{\Curve'}(1,\sigma)=|G|(q+1).
\end{gather*}
By Theorem~\ref{T: upper bound, covers},
applied to each summand $N_{\Curve'}(1,\sigma)$ for $\sigma\ne\tau$,
we obtain for each $\tau\in G$,
$$
N_{\Curve'}(1,\tau)=|G|(q+1)-\sum_{\sigma\ne\tau}N_{\Curve'}(1,\sigma)\geq q-(|G|-1)(2g'+1)\sqrt q+|G|.
$$
Let $H$ be the subgroup of $G$ of covering transformations that act trivially on~$\Curve$.
By the same reasoning as above for~$\P^1$,
we obtain
$$
\sum_{\sigma\in H}N_{\Curve'}(1,\sigma)=|H| N_\Curve(1).
$$
It follows that
$$
N_\Curve(1)\geq q-(|G|-1)(2g'+1)\sqrt q+|G|.
$$
Combined with the upper bound of Theorem~\ref{T: upper bound},
we deduce the Riemann hypothesis for $\Curve$ by Lemma~\ref{L: RH}.
\end{proof}

\section{Comparison with the Riemann Hypothesis}
\label{S: speculations}

The field of functions on a curve is analogous to the field $\Q$ of rational numbers.
However,
since $\Q$ has no field of constants,
there is no analogue of the points on~$\Curve$.
Therefore there is no Frobenius flow as in Section \ref{S:graph of F} on $\spec\Z$,
even though in every Galois extension of $\Q$ there are local Frobenius automorphisms associated with every prime number,
 as in Section \ref{S: point frobenius}.
Moreover,
since the numbers $\log|x|$,
 for $x\in\Q$,
are dense on the real line,
we cannot separate the ``point counting function'' in the different degrees.
Instead,
we have the function $\psi(x)=\sum_{p^k\leq x}\log p$.
The analogue of this function for $\Curve$ is
$$
\psi_\Curve(x)=\sum_{n\leq\log_q x}N_\Curve(n)=\sum_{k\deg v\leq\log_qx}\deg v,
$$
which counts the points on $\Curve$ defined over ${\F_{q^n}}$ with multiplicity $[(\log_qx)/n]$.
To obtain an explicit formula for $\psi_\Curve$,
we use
$
N_\Curve(n)=q^n-\omega_1^n-\dots-\omega_{2g}^n+1
$
to obtain
\begin{gather*}
\psi_\Curve(x)=
\frac{q^{[\log_qx]}-1}{1-q^{-1}}-\sum_{\nu=1}^{2g}\frac{\omega_\nu^{[\log_qx]}-1}{1-\omega_\nu^{-1}}+[\log_qx].
\end{gather*}
Using the Fourier series
$$
q^{-c\{x\}}=(1-q^{-c})\sum_{n\in\Z}\frac{e^{2\pi inx}}{c\log q+2\pi in}\quad\text{and}\quad
\{x\}=\frac12-\sum_{n\neq0}\frac{e^{2\pi inx}}{2\pi in}
$$
we obtain a formula for $\psi_\Curve(x)$ as a sum over the zeros and poles of $\zeta_\Curve(s)$,
\begin{multline*}
\psi_\Curve(x)=\frac{1}{\log q}\Biggl(\sum_{n=-\infty}^\infty\frac{x^{1+2\pi in/\log q}}{1+2\pi in/\log q}
-\sum_{\nu=1}^{2g}\sum_{n=-\infty}^\infty\frac{x^{\rho_\nu+2\pi in/\log q}}{\rho_\nu+2\pi in/\log q}\\
+\sum_{n\neq0}\frac{x^{2\pi in/\log q}}{2\pi in/\log q}+\log x\Biggr)
-\frac12-\frac1{1-q^{-1}}+\sum_{\nu=1}^{2g}\frac1{1-\omega_\nu^{-1}},
\end{multline*}
where $\rho_\nu=\log_q\omega_\nu$.
This formula should be compared with the explicit formula \cite[Thm.\ 29, p. 77]{Ingham},
$$
\psi(x)=x-\sum_\rho\frac{x^\rho}\rho-\frac{\zeta'}\zeta(0)+\sum_{n=1}^\infty\frac{x^{-2n}}{2n},
$$
which expresses $\psi(x)$ as a sum over the zeros and poles of the Riemann zeta function.
Indeed,
$1+2\pi in/\log q$ and $2\pi in/\log q$ run over all poles of $\zeta_\Curve$ (for $n\in\Z$),
and the numbers $\rho_\nu+2\pi in/\log q$ run over all zeros of this function.
Table \ref{T: comparison} compares the Riemann zeta function with the zeta function of $\Curve$.

\begin{table}[htb]
\caption{Comparison of $\spec\Z$ and $\Curve$.}
\label{T: comparison}
\begin{center}\medskip\begin{tabular}{ll}
Rational numbers	&Function fields\\
\hline\\[-9pt]
$\zeta_\Z(s)=\pi^{-s/2}\Gamma(s/2)\sum_{n=1}^\infty n^{-s}$
			&$\zeta_\Curve(s)=q^{s(g-1)}\sum_{\calD\geq0}|\calD|^{-s}$\\[2pt]
$\phantom{\zeta_\Z(s)}=\pi^{-s/2}\Gamma(s/2)\prod_{p}\frac1{1-p^{-s}}$
			&$\phantom{\zeta_\Curve(s)}=q^{s(g-1)}\prod_v\frac1{1-q_v^{-s}}$\\[5pt]
\em Simple poles\\
at $1$, residue $1$	&at $1+k\frac{2\pi i}{\log q}$,	res.\ $\frac h{(q-1)\log q}$\\
at $0$, residue $-1$	&at $k\frac{2\pi i}{\log q}$, res.\ $-\frac h{(q-1)\log q}$\\[5pt]
\em Zeros\\
at $\rho_n=\frac12+i\gamma_n$, $n\in\Z$
			&at $\rho_\nu=\frac12+i\gamma_\nu+k\frac{2\pi i}{\log q}$\\
			&($1\leq\nu\leq2g$, $q^{\rho_\nu}=\omega_\nu$)\\[5pt]
\multicolumn{2}{l}{\em Frobenius as symmetries of a cover}\\
For every extension of $\Q$
			&For every cover of $\Curve$\\[5pt]
\em Frobenius flow\\
See Remark \ref{R: shift}&Acting on points of $\Curve$\\[5pt]
\em Point counting function\\
Unknown			&$N_\Curve(n)=|\Curve(\F_{q^n})|$\\[5pt]
\em Prime counting function\\
$\psi(x)=\sum_{p^k\leq x}\log p$
			&$\psi_\Curve(x)=\sum_{k\deg v\leq\log_qx}\deg v$\\[5pt]
\multicolumn{2}{l}{\em Riemann hypothesis: $\gamma_n$ is real}\\
$\Longleftrightarrow\psi(x)=x+O(x^{1/2+\varepsilon})$
			&$\Longleftrightarrow\psi_\Curve(x)=\frac{1}{1-q^{-1}}q^{[\log_qx]}$\\
			&\hspace{65pt}$+O(x^{1/2+\varepsilon})$\\
			&$\Longleftrightarrow N_\Curve(n)=q^n+O(q^{n/2})$\\
$\Longleftrightarrow\psi(x)\leq x+O(x^{1/2+\varepsilon})\quad$
			&$\:\Longrightarrow\psi_\Curve(x)\leq\frac{1}{1-q^{-1}}q^{[\log_qx]}$\\
			&\hspace{65pt}$+O(x^{1/2+\varepsilon})$\\
			&$\:\Longrightarrow N_\Curve(n)\leq q^n+O(q^{n/2})$
\end{tabular}\end{center}
\end{table}

\begin{remark}\label{R: shift}
As is pointed out in \cite[Remark~c, p.~72]{Connes},
even though the Frobenius flow for the rational numbers is not known,
the dual algebraic picture obtained from class field theory is complete.
Thus the Frobenius flow on $\Curve$ corresponds to the action of $q$ in the idele class group $\A^*/K^*$ on the space of adele classes $\A/K^*$,
 and the Frobenius flow on $\spec\Z$ corresponds to the action of $\R^{>0}\subset\A^*/\Q^*$ on the space~$\A/\Q^*$.
In additive language,
according to Deninger~\cite{cD93,cD93a},
the analogue of the Frobenius flow of the first step may be provided by the shift on the real line.
\end{remark}

By the last entry of Table~\ref{T: comparison},
it is only necessary to prove the upper bound of Theorem \ref{T: upper bound} to obtain the Riemann hypothesis,
and establishing the lower bound of Section \ref{S: point frobenius} becomes unnecessary.\footnote
{This is a consequence of the density of $\log|x|$,
$x\in\Q$,
in the real line.}
Therefore we need the analogue for $\spec\Z$ of the first inequality,
$$
N_\Curve(1)\leq q+O(\sqrt q).
$$
This means that we only need to translate the part of Bombieri's proof that depends on the Frobenius flow on $\Curve$,
not the part that depends on the Frobenius symmetries of covers of $\Curve$.
To translate the argument of Section~\ref{S:graph of F},
we could try to construct a polynomial that vanishes at all prime numbers up to a certain bound.
The infinitesimal generator of the shift on the real line,
which is the counterpart of the Frobenius flow by Remark \ref{R: shift},
 is the derivative operator.
So we may try to construct such a function that vanishes at the prime numbers to a high order.
Then we want to bound the degree of this polynomial.

The ``degree'' of a rational prime number is $\log p$,
so a prime number should correspond to $\log p$ points on some curve.
In Nevanlinna Theory (see,
for example,
\cite{Hayman,LC}),
the counting function of zeros of a meromorphic function $f$ in the disc of radius $r$ is
$$
N_f(0,r)=\sum_{|x|<r\colon f(x)=0}\ord_x(f)\log\frac{r}{|x|},
$$
so one could consider a function such as
$$
f(z)=\prod_{p\leq x}(1-pz)^{[\log_px]}
$$
 on the unit disc.
For this function
$$
N_f(0,1)=\psi(x).
$$

Trying to copy Bombieri's proof,
we could take $s_i=z^{1-i}$ for $i\geq1$ as the basis of functions that have only a pole at `infinity',
and the coefficients are functions on~$\spec\Z$ with a pole at $v_\infty$ alone,
that is,
the coefficients are integers $b_i$.
Thus\footnote
{If this series is infinite,
 it does not converge for $|z|\leq1$,
so we may have to require that it has an analytic continuation.}
$$
f(z)=\sum_{i=1}^\infty b_iz^{1-i}.
$$
Note that the coefficients are determined by $f$,
that is,
the analogue of Lemma \ref{L: f|phi} is automatically satisfied.
We do not know what the analogue of the choice $b_i=a_i^{p^\mu}$ could be.

The main problem is that the arithmetic coordinate cannot be compared with the geometric coordinate:
$z\in\C$ is the geometric coordinate and the coefficients of~$f$ are integers,
i.e.,
functions on the arithmetic coordinate $\spec\Z$.
Hence there is no diagonal and we have to force the vanishing of $f$ at the primes in an artificial manner.
As a consequence,
we do not know how to bound the Nevanlinna height (i.e.,
the degree) of this function.\medskip

It is interesting to pursue this idea a little further.
For $r>0$,
let $\Delta_r$ be the disc of radius $r$ with boundary $\Gamma_r$,
positively oriented.
Let
\begin{gather}\label{E:c}
f(z)=cz^{\ord(f,0)}+\dots
\end{gather}
be a meromorphic function with leading coefficient $c$ in its Laurent series at $0$.
Nevanlinna theory starts with the Poisson--Jensen formula,
which we interpret as a sum over all valuations of the field of meromorphic functions on $\Delta_r$,
\begin{gather}\label{E:PJ}
v_0(f)\log r+\sum_{0<|x|<r}v_{x}(f)\log\frac r{|x|}+\int_{\Gamma_r}v_z(f)\,\frac{dz}{2\pi iz}=-\log|c|,
\end{gather}
where the valuations are $v_z(f)=-\log|f(z)|$ for each $z$ on the boundary of the disc of radius $r$,
and $v_{x}(f)=\ord(f,x)$ for each $x$ inside the disc.\footnote
{Note that $v_0\log r$ is only a valuation for $r\geq1$ and that it is the trivial valuation for $r=1$.}

Note that the sum \eqref{E:PJ} of the valuations is constant,
depending on the function but not on $r$.
This should be compared to the fact that $\sum_vv(\alpha)\deg v=0$ for every nonzero function in the function field of $\Curve$.
It is a mystery why in Nevanlinna theory,
the sum over all valuations does not necessarily vanish.
However,
if we consider the subfield of functions for which the Laurent series has rational coefficients,
then we can write the Poisson--Jensen formula as,
with $c\in\Q$ as in \eqref{E:c},
$$
\sum_{p\neq\infty}v_p(c)\log p+v_0(f)\log r+\sum_{0<|x|<r}v_{x}(f)\log\frac r{|x|}+\int_{\Gamma_r}v_z(f)\,\frac{dz}{2\pi iz}=0,
$$
where the first sum is over all $p$-adic valuations of $\Q$.
It is still puzzling why the archimedean valuation $v_\infty(c)=-\log|c|$ should be excluded from this sum.
We interpret this as meaning that the arithmetic and geometric coordinates (in the sense of Remark \ref{R:arith geo}) meet at $v_{0}$ (geometrically) and $v_\infty$ (arithmetically).
Since this is not a double point,
we only see $v_{0}$.
Note that the archimedean valuations have all been pushed to the boundary of the disc,
and that they are the only nondiscrete valuations.

\begin{remark}
For large $r$,
the archimedean valuations are essentially nonarchimedean,
since for two meromorphic functions $f$ and $g$ that are not a constant multiple of each other,
$|f(z)+g(z)|$ will be close to $\max\{|f(z)|,|g(z)|\}$ on most of the circle~$\Gamma_r$,
and only on small portions of $\Gamma_r$ will $f(z)$ and $g(z)$ be comparable in size.
More specifically,
if the only defects (in the sense of Nevanlinna theory) of $f/g$ are among $\infty$, $0$, $-1$, $e^{2\pi i/3}$ and $e^{4\pi i/3}$,
then the archimedean valuations behave like nonarchimedean valuations for large $r$.
\end{remark}

It seems that by Nevanlinna theory we obtain a connection,
albeit a rather loose one,
between the geometric valuations $v_z$ and $v_{x}$,
and the arithmetic valuations~$v_p$,
and seemingly,
the geometric and arithmetic coordinates intersect at $v_{0}$.
We consider this the point `at infinity',
as in the exposition of Bombieri's proof of the Riemann hypothesis for a curve $\Curve$.
We invite the reader to continue this line of reasoning.

\subsubsection*{Acknowledgements}

We thank Enrico Bombieri for information about the preparation of his Bourbaki seminar talk \cite{Bo},
and Alain Connes for his encouragement.

\end{document}